\documentclass[11pt,a4paper]{amsart}

\pdfoutput=1

\usepackage{amssymb}
\usepackage{amsmath} 
\usepackage{mathrsfs}     
\usepackage{verbatim}   
\usepackage{graphicx}   
\usepackage{enumitem}
\usepackage{color}     
\usepackage{hyperref}

\numberwithin{equation}{section}

\newtheorem{theorem}{Theorem}[section]
\newtheorem{lemma}[theorem]{Lemma}
\newtheorem{proposition}[theorem]{Proposition}
\newtheorem{corollary}[theorem]{Corollary}

\theoremstyle{definition}

\newtheorem{definition}[theorem]{Definition}
\newtheorem{remark}[theorem]{Remark}

\newcommand\C{\mathbb{C}}

\newcommand\D{\mathbb{D}}

    \usepackage{enumitem}
    \setenumerate[1]{label=(\arabic*)}

\newcommand{\N}{{\mathcal N}}

\def\theta{\vartheta}

\def\S{{\mathbb{S}}}

\def\D{\mathbb{D}}
\def\Cc{\widehat{{\C}}}

\def\ol{\overline}

\begin{document}

\title{Combinatorial properties of Newton maps}

\author[Lodge]{Russell Lodge}
\email{russell.lodge@indstate.edu}
\address{Department of Mathematics and Computer Science, Indiana State University, 200 North Seventh Street, Terre Haute, IN 47809, USA}

\author[Mikulich]{Yauhen Mikulich}
\email{y.mikulich@gmail.com}
\address{All\'ee Leotherius 2, 1196 Gland, Switzerland}

\author[Schleicher]{Dierk Schleicher}
\email{dierk.schleicher@univ-amu.fr}
\address{Aix–Marseille Université, Institut de Mathématiques de Marseille, 163 Avenue de Luminy, 13009 Marseille, France}

\subjclass[2010]{Primary 30D05, 37F10, 37F20}
\thanks{This research was partially supported by the Deutsche Forschungsgemeinschaft (DFG), as well as the advanced grant 695621 “HOLOGRAM” of the European Research Council (ERC), which is gratefully acknowledged. The authors would also like to thank Johannes R\"uckert and Kostiantyn Drach for useful conversations, and the anonymous referees for helpful suggestions on an earlier version of this paper. }

\begin{abstract}
This paper constructs a combinatorial model for all postcritically finite rational maps arising as the Newton's method of a complex polynomial.  This model is used in \cite{LMS2} to give a combinatorial classification of postcritically finite Newton maps of any degree.  
\end{abstract}

\date{\today}
\maketitle

\tableofcontents

\section{Introduction} \label{Sec_Overview}

The dynamical properties of rational functions $f:\widehat{\mathbb{C}}\rightarrow\widehat{\mathbb{C}}$ have been intensely scrutinized over the last few decades, though in some ways the remarkable theory which has emerged is only in its early stages.  Natural motivation for the topic comes from the study of Newton's root finding method applied to a complex polynomial.  For instance, it has long been observed that in some cases Newton's method does not converge to a root for open sets of initial values in $\mathbb{C}$; Smale posed the problem of ``systematically finding" those polynomials whose Newton's method have such open sets \cite[Problem 6]{Smale85}.  In a different vein, a number of studies have been carried out on Newton's method as a root-finding method \cite{McM,HSS,schleicher_2002,Schl,Sch08,BLS,SchSt,BAS,SCRSSS20}.

Finite combinatorial models have been successfully created to encode the dynamics of postcritically finite complex polynomials \cite{BFH,Poirier}, but similar attempts for rational maps have met with formidable difficulties (postcritically finite maps are chosen for study because they are structurally significant in parameter space, and because Thurston's characterization and rigidity theorem is available).  This paper will produce a combinatorial invariant that will yield a classification of \emph{all} postcritically finite Newton maps worked out in \cite{LMS2}.  No other combinatorial classification of this scope exists for non-polynomial rational maps, as explicit classifications have only been made in the past for one-dimensional families. 

\begin{definition}[Newton map]\label{defn:NewtonMap}
A rational function $f:\widehat{\C}\to\widehat{\C}$ of degree $d\geq 3$ is called a \emph{Newton map} if there is some complex polynomial $p(z)$ so that $f(z)=z-\frac{p(z)}{p'(z)}$ for all $z\in\mathbb{C}$.
\end{definition}

The \emph{Newton map of $p$} is given by $N_p(z)=z-\frac{p(z)}{p'(z)}$, and it should be observed that $N_p$ arises naturally when Newton's method is applied to find the roots of $p$.  The cases $d<3$ are excluded in this paper because they are trivial.  Each root of $p$ is an attracting fixed point of $N_p$, and the point at infinity is a repelling fixed point of $N_p$.  The degree $d$ coincides with the number of distinct roots of $p$.  If $N_p$ is postcritically finite, the finite fixed points of $N_p$ must be superattracting, which implies that all roots of $p$ are simple.

In this paper, we construct a finite forward invariant graph for $N_p$ called an \emph{extended Newton graph}.  We then give an axiomatic definition of the class of graphs called ``abstract extended Newton graphs" (see Definition \ref{Def_AbstractExtNewtGraph}) and show that our graphs satisfy these axioms.  In \cite{LMS2} we show the converse: every abstract extended Newton graph comes from a postcritically finite Newton map.  This leads to a combinatorial classification of postcritically finite Newton maps up to affine conjugacy in terms of abstract extended Newton graphs with an appropriate equivalence relation.   Foundational to both articles will be the results in \cite{MRS} which gives a classification of all postcritically \emph{fixed} Newton maps, namely those Newton maps whose critical points are all mapped onto fixed points after finitely many iterations. Even though postcritically fixed Newton maps were the immediate concern, it must be emphasized that many results in \cite{MRS} were formulated for the more general class of ``attracting critically finite Newton maps'' which includes all postcritically finite Newton maps. This was done in anticipation of the present work.

We give a brief overview of the graph invariant that will be used to classify postcritically \emph{finite} Newton maps in \cite{LMS2}. If $N_p$ is a postcritically finite Newton map, then as in \cite{MRS}, we define the channel diagram $\Delta$ of $N_p$ to be the union of the accesses from finite fixed points of $N_p$ to $\infty$ (see Section \ref{Sec_NewtonGraph}).   Next, using the notation $N_p^{n}:=N_p^{\circ n}$, the Newton graph of level $n$ is constructed to be the connected component of $N_p^{-n}(\Delta)$ containing $\infty$ and is denoted by $\Delta_n$.  For a sufficiently high level $n$, the Newton graph captures the behavior of critical points mapping onto fixed points. See Figure \ref{Fig_PreFixedExample} for an example of a Newton graph of level one.

We call a critical point \emph{free} if it is not contained in the Newton graph $\Delta_n$ for any level $n$; put differently, a critical point is free if its forward orbit does not contain a fixed point.  See Figure \ref{Fig_Deg4SmallBasilicas} for an example of a Newton map with free critical points. We now describe the combinatorial objects that capture the behavior of free critical points.

\begin{figure}[h]
\centerline{\includegraphics[width=135mm]{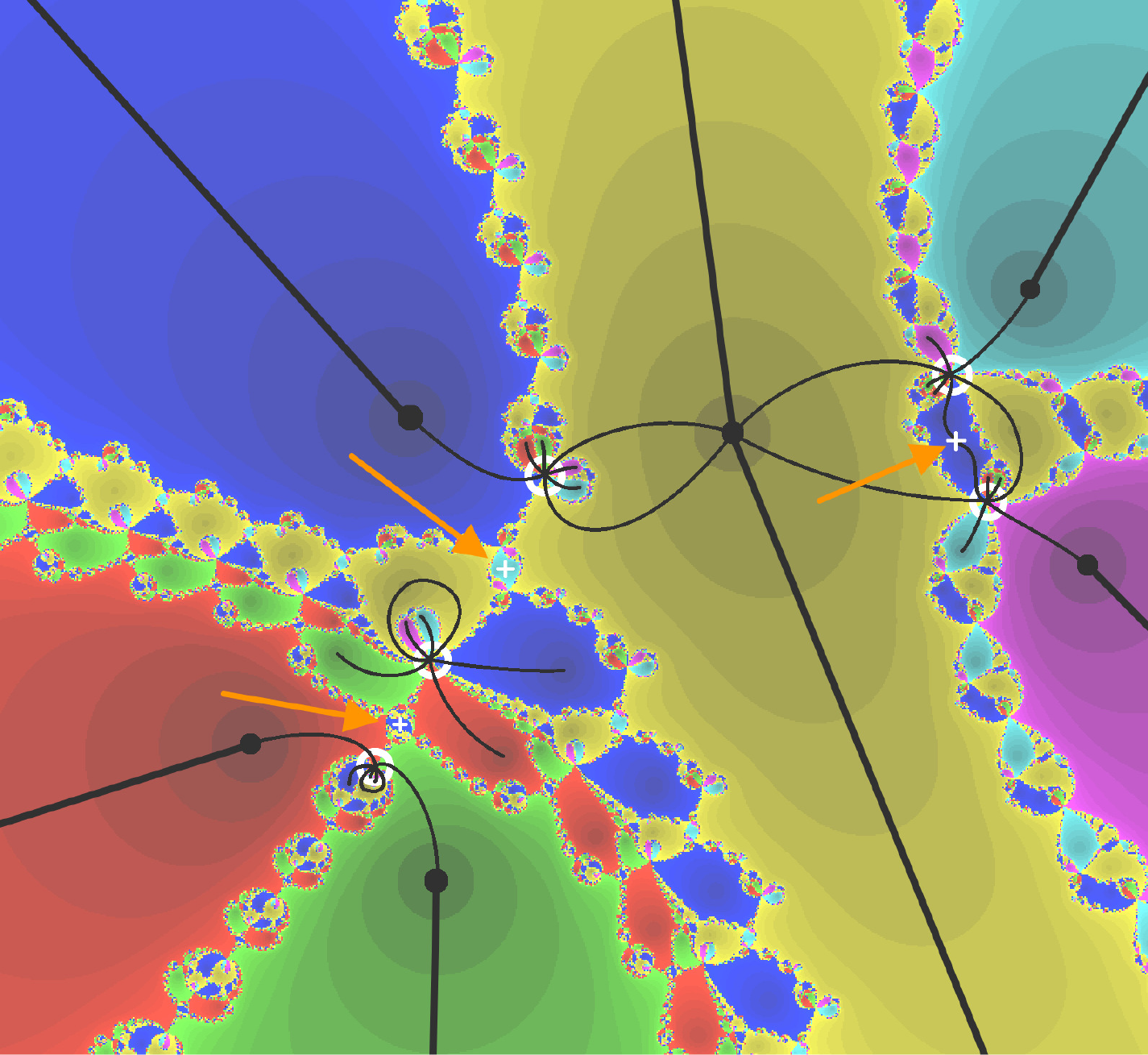}}
\caption{Dynamical plane of a degree 6 Newton map $N_p$. Roots are indicated by black dots. The channel diagram $\Delta$ is drawn with thick black curves, and $N_p^{-1}(\Delta)\setminus\Delta$ is drawn with thin black curves. Poles are indicated by white circles, where one pole clearly does not lie on the boundary of the immediate basin of a root. The Newton graph of level one $\Delta_1$ is the component of $N_p^{-1}(\Delta)$ that does not intersect this special pole.  There are three non-fixed critical points that are simple, each indicated by a white ``+'' and an orange arrow. There are no free critical points.}
\label{Fig_PreFixedExample}
\end{figure}

For each periodic postcritical point of $N_p$ having period greater than one (i.e.\ a periodic postcritical point that isn't $\infty$ or a root of $p$), we use the renormalization result of \cite{DLSS} to produce a local model using extended Hubbard trees.  To capture the behavior of critical points that map into a Hubbard tree after some number of iterates, appropriate preimages are taken of these Hubbard trees.

Thus far, all postcritical points are contained in either the Newton graph or one of the Hubbard tree (preimages), but the Hubbard trees are disjoint from the Newton graph.  To remedy this, ``Newton rays" are used to connect certain repelling periodic points on the extended Hubbard trees to the Newton graph (see Section \ref{Sec_Newtrays}).  Each Newton ray is either an internal ray in the immediate basin of a root, or is comprised of infinitely many preimages of edges of the Newton graph. 

\begin{figure}[h]
\centerline{\includegraphics[width=130mm]{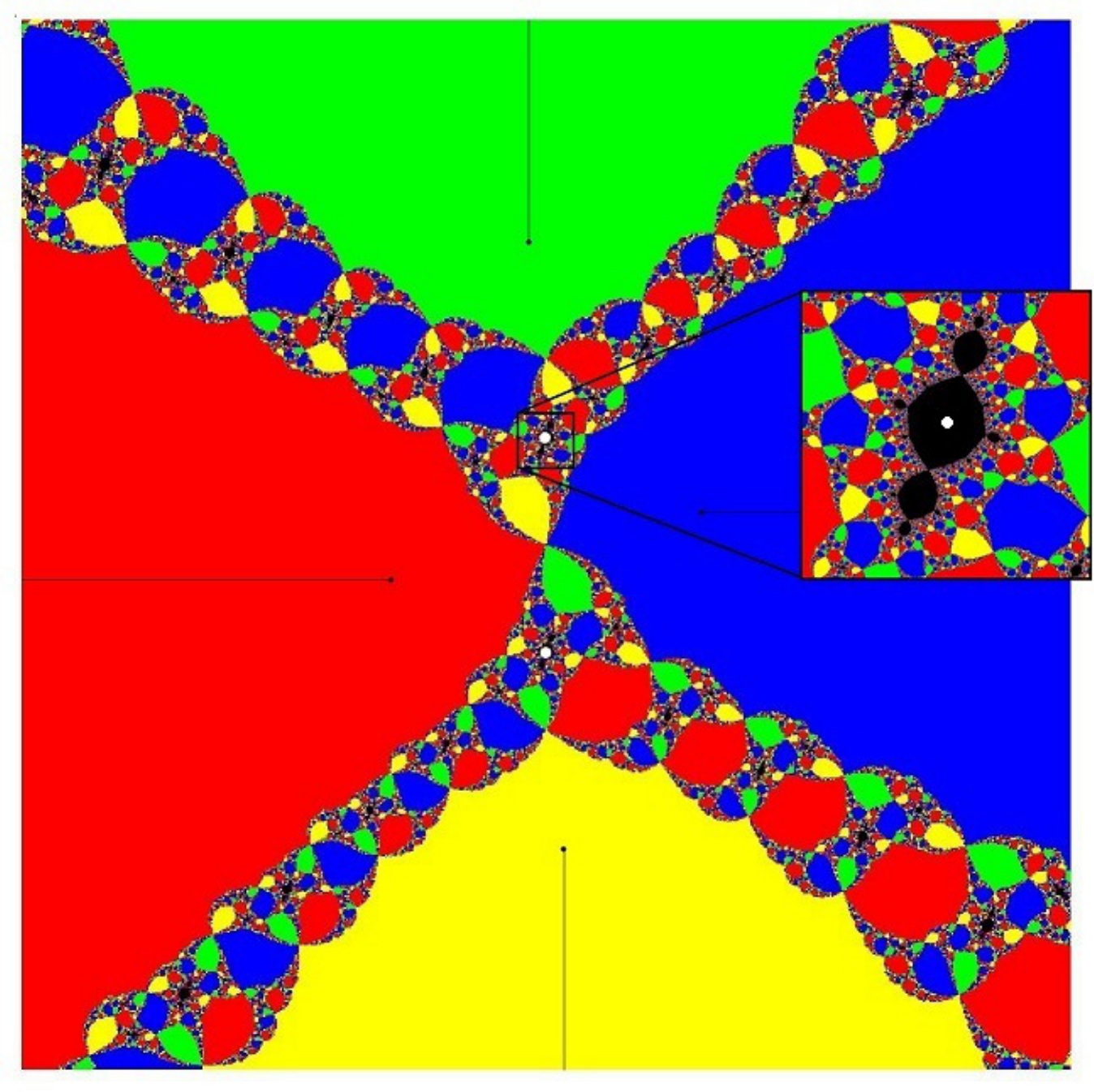}}
\caption{Part of the dynamical plane of the Newton map of degree 4 (with an inset zoom) for the monic polynomial with roots given approximately by $\pm(.593+.130i)$ and $\pm(-.0665+1.157i)$.  The roots are drawn as four black dots, and the channel diagram is indicated by the thin black lines.  The two white dots represent free critical points at $\pm.408i$, and they lie in a common four-cycle contained in the two black ``little basilicas''. The basin of these free critical points is an open set of starting points of Newton's method that do not converge to roots (image by K. Mamayusupov).}
\label{Fig_Deg4SmallBasilicas}
\end{figure}

Now the \emph{extended Newton graph}, denoted $\Delta^*_\N$, can be produced for $N_p$.  It is a finite graph composed of:
\begin{itemize}
\item the Newton graph
\item the Hubbard tree pieces for each free critical point of $N_p$ 
\item Newton rays connecting each Hubbard tree piece to the Newton graph.  
\end{itemize}
Restriction of $N_p$ to $\Delta^*_\N$ yields a self map, and the graph together with this self map is denoted $(\Delta^*_\N,N_p)$.

The axioms for an \emph{abstract extended Newton graph} are given in Definition \ref{Def_AbstractExtNewtGraph}, and the following theorem is proved.
\begin{theorem}[Newton maps generate extended Newton graphs] \label{Thm_NewtMapsGenerateExtNewtGraphs}
For any extended Newton graph $\Delta^*_\N\subset\widehat{\mathbb{C}}$ associated to a postcritically finite Newton map $N_p$, the pair $(\Delta^*_\N,N_p)$ satisfies the axioms of an abstract extended Newton graph. 
\end{theorem}

It will be shown in \cite{LMS2} that every abstract extended Newton graph is realized by a unique postcritically finite Newton map up to affine conjugacy.  This result will be used to establish a bijection between the set of postcritically finite Newton maps up to affine conjugacy and the set of abstract extended Newton graphs up to some explicit equivalence. 

\subsection{Structure of this paper}  Section \ref{Sec_ExistingResults} introduces basic properties of Newton maps for later use, as well as a brief history of existing combinatorial models for Newton maps.

Section \ref{Sec_NewtonGraph} constructs the Newton graph edges of the extended Newton graph.  In so doing, the notions of a channel diagram, Newton graph and their abstract counterparts are defined.  Extensions of certain graph maps to a branched cover of the 2-sphere is also discussed.

Section \ref{Sec_RenormalizationNewton} constructs the Hubbard tree edges of the extended Newton graph.  Preliminaries on extended and abstract extended Hubbard trees are covered in \ref{Sec_HubbardTrees} and \ref{Sec_PolynomialLike}.   The renormalization result for Newton maps in \cite{DLSS} is introduced in Section \ref{Sec:RenormOfNewtonMaps} and the Hubbard trees are constructed.

Section \ref{Sec_Newtrays} initiates the construction of Newton ray edges, which will connect the Newton graph with fixed points of the polynomial-like mappings arising from renormalization.  An ordering is placed on the rays to enable well-defined choices among the rays landing at a single fixed point. 

Section \ref{Subsec_ConstructionExtNewtGraph} combines the three types of edges to produce the extended Newton graph.  An example of such a graph is given in \ref{Subsec_ExamplesExtNewtGraphs}.

Section \ref{Sec_NewtMapsGenerateExtNewtGraphs} defines the abstract analog of Newton rays and extended Newton graphs, and shows that an extended Newton graph constructed for a postcritically finite Newton maps satisfies the abstract definition.  The main result of the paper (Theorem \ref{Thm_NewtMapsGenerateExtNewtGraphs}) is proven.\\

\section{Known results about Newton maps}\label{Sec_ExistingResults}

This section will catalog some well-known properties of Newton maps for later use. A brief history of the various combinatorial models for Newton maps will be given as well.

The following characterization of Newton maps in terms of fixed point multipliers is essentially due to Janet Head (Proposition 2.1.2 \cite{He}).

\begin{proposition}[{\cite[Cor 2.9]{RS}}]
A rational map $f$ of degree $d\geq 3$ is a Newton map if and only if $\infty$ is a repelling fixed point of $f$ and for each of the other fixed points $\xi\in\widehat{\C}$, there is an integer $m\geq 1$ so that $f'(\xi)=(m-1)/m$.
\end{proposition}

Let $p$ be a monic polynomial of degree $d$ with complex coefficients and simple roots $a_1, a_2,...,a_d$.  Define the Newton map corresponding to $p$ by
\begin{equation} \label{Eq_NewtonMapDefinition}
N_p=id - \frac{p}{p'}.
\end{equation}  One can see from the equation \[
N'_p= \frac{p\cdot p''}{(p')^2}
\]  
that the roots of $p$ are attracting fixed points of $N_p$.
The point at infinity is a repelling fixed point of $N_p$ with multiplier $d/(d-1)$.

Note that the roots of $p$ must be simple for the purposes of this study because otherwise the corresponding Newton map would have an attracting fixed point that is not superattracting, and would thus not be postcritically finite.  The map $N_p$ has degree $d$, and its $d+1$ fixed points are given by the roots $a_1,a_2,...,a_d,\infty$; thus all finite fixed points of the Newton map are critical.

Shishikura \cite{Sh} proved that the Julia set of a rational map is connected if there is only one repelling fixed point.  Combining this with the facts just mentioned, he obtains the following.

\begin{proposition}\label{prop:Jconnected}
The Julia set of a Newton map $N_p$ is connected.
\end{proposition}

The Fatou components of $N_p$ that contain roots of $p$ play a foundational role in our combinatorial constructions.

\begin{definition}[Immediate basin] Let $N_p$ be a Newton map and $\xi \in \C$ a finite
fixed point of $N_p$. Let $\displaystyle B_{\xi} = \{ z \in \C: \lim_{n\to\infty}
N_p^n(z)=\xi \}$ be the basin (of attraction) of $\xi$. The connected component of $B_\xi$ containing $\xi$ is
called the \emph{immediate basin} of $\xi$ and denoted $U_\xi$.
\end{definition}

It was shown by Przytycki that $U_\xi$ is simply connected and unbounded \cite{Pr}. 

\begin{definition}[Invariant access to $\infty$] Let $\xi$ be an
attracting fixed point of $N_p$ and $U_\xi$ its immediate basin. An
\emph{access} of $\xi$ to $\infty$ is a homotopy class of curves
in $U_\xi$ that begin at $\xi$, land at $\infty$ and are
homotopic in $U_\xi$ with fixed endpoints.
\end{definition}

Let $m_{\xi}$ be the number of critical points of a Newton map $N_p$
in the immediate basin $U_{\xi}$, counted with multiplicity. Then
$N_p|_{U_{\xi}}$ is a branched cover of degree $m_{\xi}+1$. The following proposition is used to produce the first-level combinatorial data for Newton maps.

\begin{proposition}[Accesses to infinity {\cite{HSS}}] \label{Prop:AccessesHSS}   The immediate basin $U_{\xi}$ has exactly $m_{\xi}$ accesses to $\infty$.
\end{proposition}

Let $f:\S^2 \to \S^2 $ be an orientation-preserving branched
cover of degree greater than one. Denote the local degree of $f$ at a point $z$ by $\deg_z f$.

\begin{definition} \label{Def_PCFNewtonMaps}
Set $C_f = \{ critical \ points \ of \ f \} = \{x|\deg_x f>1\}$ and
\[
P_f=\bigcup_{n\ge1} f^n(C_f).
\]
The map $f$ is called a
\emph{postcritically finite branched cover} if $P_f$ is finite.
We say that $f$ is \emph{postcritically fixed} if for each $x\in C_f$, there exists $n>0$ such that $f^{n}(x)$ is a fixed point of $f$.
\end{definition}

Combinatorial models for various types of postcritically finite Newton maps exist.  Janet Head introduced the ``Newton tree" to characterize postcritically finite cubic Newton maps \cite{He}. Tan Lei built upon these ideas to give a classification of postcritically finite cubic Newton maps in terms of matings and captures \cite{TL}.  Tan Lei also gave another combinatorial classification of the Newton cubic family using abstract graphs. More precisely, every postcritically finite cubic Newton map gives rise to a forward invariant finite connected graph that satisfies certain axioms.  Conversely, every graph which satisfies these axioms is realized by a unique postcritically finite cubic Newton map using Thurston's theorem.  Finally, the graph associated to a postcritically finite cubic Newton map is realized by the same cubic Newton map under Thurston's theorem (all graphs and rational maps are considered up to the natural equivalences).

Fewer results exist for higher degree.  Jiaqi Luo studies Newton maps of arbitrary degree with exactly one non-fixed critical value, which we call ``unicritical Newton maps". For such maps, Luo constructs a forward-invariant, finite topological graph analogous to the Newton graph of this paper. In the spirit of Tan Lei's work, he defines a ``topological Newton map" to be a branched cover with the same critical orbit properties as a unicritical Newton map, and then shows that Thurston obstructions for topological Newton maps may only be Levy cycles of a special type \cite{Lu2}.  Assuming further that a topological Newton map satisfies certain explicit conditions on the attracting basins of the fixed critical points, Luo proves that no Thurston obstructions exist if the non-fixed critical value is either periodic or contains a fixed critical point in its orbit \cite{Lu}.  

Using a very different approach, \cite[Section 11]{CGNPP} describes a process by which Newton maps whose critical points are all fixed may be produced by ``blowing up" the edges of a multigraph. In fact, for any degree, the maps with only fixed critical points have been classified in the rational \cite{HClass} and anti-rational case \cite{Gey, LLM}.

In his thesis, Johannes R\"uckert \cite{RueckertThesis} classified all \emph{postcritically fixed} Newton maps for arbitrary degree (the results are also found in \cite{MRS}). For every postcritically fixed Newton map, a connected forward-invariant finite graph that contains the whole postcritical set is constructed.  The notion of an ``abstract Newton graph" is introduced, and it is seen that the forward-invariant graph just described is in fact an abstract Newton graph. It is shown that each abstract Newton graph is realized by a unique postcritically fixed Newton maps, and that the abstract graphs give the classification. The results of \cite{DLSS} extend the Newton graph construction to \emph{postcritically finite} Newton maps, and use puzzle partitions to establish combinatorial properties of the Newton graph that are essential to the present work. Both the present work and \cite{LMS2} are based on the thesis of \cite{Mik}.\\

\section{Newton graphs from Newton maps} \label{Sec_NewtonGraph}

Some preliminaries about graph maps are presented, following \cite[Chapter 6]{BFH}.  In particular, a condition under which a graph map may be uniquely extended to a branched cover of the whole sphere is presented which will be useful for the definition of the abstract extended Newton graph.  The following is the so-called ``Alexander trick" which is fundamental to such extension results.

\begin{lemma}
\label{Lem:AlexanderTrick} Let $h:\S^1 \to \S^1$ be an
orientation-preserving homeomorphism. Then there exists an orientation
preserving homeomorphism $\ol{h}:\overline{\D} \to \overline{\D}$ such that $\ol{h}|_{\S^1}
=h$. The map $\ol{h}$ is unique up to isotopy relative $\S^1$.
\end{lemma}
\begin{definition}[Finite graph] A \emph{vertex} is a point in $\S^2$. Let $V$ be a finite set of distinct vertices. An \emph{edge} is a subset of $\S^2$ of the form $\lambda(I)$ where $I=[0,1]$ and
\begin{itemize}
 \item $\lambda:I\to\S^2$ is continuous and injective on $(0,1)$, and
\item $\lambda(x)\in V \iff x \in \partial I$.
\end{itemize}
Let $E$ be a finite set of edges that (pairwise) intersect only at vertices.
A \emph{finite graph} (in $\S^2$) is a pair of the form $(V,E).$
\end{definition}
 We sometimes omit the reference to the ambient space $\mathbb{S}^2$ though it is always implicit.

\begin{definition}[Subgraphs]
Let $\Gamma_1=(V_1,E_1)$ and $\Gamma_2=(V_2,E_2)$ be finite graphs. We say that $\Gamma_1$ is a \emph{subgraph} of $\Gamma_2$ (denoted $\Gamma_1\subset\Gamma_2$) if $V_1\subset V_2$ and $E_1\subset E_2$. 
\end{definition}

\begin{definition}[Graph map]
Let $\Gamma_1,\Gamma_2$ be connected finite graphs. A continuous map $f:\Gamma_1 \to \Gamma_2$ is called a \emph{graph map} if it is injective on each edge of $\Gamma_1$, the set of vertices is forward and backward invariant, and $f$ is compatible with the graph embeddings in the sense that for each vertex $v$ in $\Gamma_1$, the map $f$ preserves the cyclic order of edges terminating at $v$.
\end{definition}

The \emph{degree} $\deg_v(f)$ of $f$ at a vertex $v$ in $\Gamma_1$ is defined to be the number of edges at $v$ that are mapped by $f$ to the same image edge at $f(v)$ (since $f$ preserves the local cyclic ordering, this definition is independent of the choice of image edge).

\begin{definition}[Regular extension]
\label{Def_GraphMap} Let $f:\Gamma_1\to\Gamma_2$ be a graph map. An
orientation-preserving branched cover $\ol{f}:\S^2\to\S^2$ is
called a \emph{regular extension} of $f$ if $\ol{f}|_{\Gamma_1}=f$
and $\ol{f}$ is injective on each component of $\S^2\setminus
\Gamma_1$.
\end{definition}
It follows that every regular extension $\ol{f}$ may have critical points only at the vertices of $\Gamma_1$, and the local degree of $\ol{f}$ coincides with $\deg_v(f)$.
\begin{lemma}[Isotopic graph maps {\em \cite[Corollary 6.3]{BFH}}]
\label{Lem_IsotopGraphMaps}  Let
$f,g:\Gamma_1\to\Gamma_2$ be two graph maps that coincide on the
vertices of $\Gamma_1$ such that for each edge $e$ in $\Gamma_1$ we have $f(e)=g(e)$ as a set. Suppose that $f$ and $g$
have regular extensions $\ol{f},\ol{g}:\S^2\to\S^2$. Then there
exists a homeomorphism $\psi:\S^2\to\S^2$, isotopic to the identity
relative the vertices of $\Gamma_1$, such that $\ol{f}=\ol{g} \circ
\psi$.
\end{lemma}

We must establish some notation for the following proposition from \cite{BFH}.  Let $f:\Gamma_1\to\Gamma_2$ be a graph map. For each
vertex $v$ of $\Gamma_i$ with fixed $i\in\{1,2\}$, choose a neighborhood $U_v\subset\S^2$ such
that all edges of $\Gamma_i$ that enter $U_v$ terminate at $v$, the vertex $v$ is the only vertex of $\Gamma_i$ in $U_v$, and the neighborhoods $U_v$ and $U_w$ are disjoint for all vertices $v\neq w$ in $\Gamma_i$. We
may assume without loss of generality that in local coordinates, $U_v$ is a round disk of radius $1$ that is centered at $v$ and that the intersection of any edge of $\Gamma_i$ with $U_v$ is either empty or a radial line segment.  Without loss of generality, we may assume that $f|_{U_v}$ is length-preserving for all vertices $v$ in $\Gamma_1$.

We describe how to explicitly extend $f$ to each $U_v$.  For a vertex $v\in\Gamma_1$, let $\gamma_1$ and
$\gamma_2$ be two adjacent edges ending there. In local coordinates,
these are radial lines at angles $\Theta_1,\Theta_2$ where
$0<\Theta_2-\Theta_1\leq 2\pi$ (if $v$ is an endpoint of $\Gamma_1$,
then set $\Theta_1=0$, $\Theta_2=2\pi$). In the same way, choose
arguments $\Theta_1',\,\Theta_2'$ for the image edges in $U_{f(v)}$
and extend $f$ to a map $\tilde{f}$ on $\Gamma_1\cup\bigcup_v U_v$
defined by
\begin{equation}\label{eqn:localExtension}
\tilde{f}(\rho,\Theta)=\left(\rho,
\frac{\Theta_2'-\Theta_1'}{\Theta_2-\Theta_1}\cdot(\Theta-\Theta_1)+\Theta_1'\right),
\end{equation}
where $(\rho,\Theta)$ are polar coordinates in the sector bounded by
the rays at angles $\Theta_1$ and $\Theta_2$. In particular, sectors are
mapped onto sectors in an orientation-preserving way.

\begin{proposition}[{\cite[Proposition 5.4]{BFH}}]
\label{Prop_RegExt}  A graph map
$f:\Gamma_1\to\Gamma_2$ has a regular extension if and only if for
every vertex $y\in\Gamma_2$ and every component $U$ of
$\S^2\setminus\Gamma_1$, the extension $\tilde{f}$ is injective on
\[
    \bigcup_{v\in f^{-1}(y)} U_v \cap U\;.
\]
\end{proposition}

The combinatorial classification of postcritically \emph{fixed} Newton maps (all critical points mapping onto fixed points after finitely many iterations) was given in \cite{MRS} using a combinatorial object called the ``Newton graph".    We give the analogous construction for a postcritically \emph{finite} Newton map, noting that the results mentioned below from \cite{MRS} were proven there in this more general context.  The graph constructed below will also be called the Newton graph.

 Let the superattracting fixed points of a postcritically finite Newton map $N_p$ be denoted by $a_1,a_2,\ldots,a_d$. Let $U_i$
denote the immediate basin of $a_i$. Then $U_i$ has a global
B\"ottcher coordinate $\phi_i:(\mathbb{D},0) \to (U_i,a_i)$ with the
property that $N_p(\phi_i(z))=\phi_i(z^{k_i})$ for each $z \in
\mathbb{D}$ (the complex unit disk), where $k_i-1 \geq 1$ is the
multiplicity of $a_i$ as a critical point of $N_p$. The map $z \to
z^{k_i}$ fixes $k_i-1$ rays in $\mathbb{D}$. Under
$\phi_i$, these map to $k_i-1$ pairwise disjoint (except for endpoints) simple curves
$\Gamma^1_i,\Gamma^2_i,\ldots,\Gamma^{k_i-1}_{i}\subset U_i$ that
connect $a_i$ to $\infty$, are pairwise non-homotopic in $U_i$
(with homotopies fixing the endpoints) and are invariant under $N_p$. They
represent all accesses to $\infty$ of $U_i$ (see Proposition \ref{Prop:AccessesHSS}). It is clear that the union
\begin{center}
$\displaystyle \bigcup_{i} \bigcup_{j=1}^{k_{i}-1}
\overline{ \Gamma^{j}_{i} }$
\end{center}
forms a connected set in $\Cc$.

\begin{definition}[Channel diagram of a Newton map]\label{Def_ConcreteChannelDiagram}
The \emph{channel diagram} $\Delta$ associated to $N_p$ is the finite connected graph with  vertex set $\{\infty, a_1, a_2,...,a_d\}$ and edge set:
\[\displaystyle \bigcup_{i} \bigcup_{j=1}^{k_{i}-1}
\Big\{\overline{ \Gamma^{j}_{i} }\Big\}.\]
\end{definition}

It follows from the definition that $N_p(\Delta)= \Delta$. The channel diagram records the mutual locations of the immediate basins of $N_p$ and provides a first-level combinatorial information about the dynamics of the
Newton map. 

\begin{definition}[Level $n$ Newton graph]\label{Def_ConcreteNewtonGraph}
For any $n\ge 0$, denote by $\Delta_n$ the connected
component of $N^{-n}_p(\Delta)$ that contains $\Delta$. The pair
$(\Delta_n,N_p|_{\Delta_n})$ is called the \emph{Newton graph} of $N_p$ at level $n$. 
\end{definition}

We next state one of the main results in \cite[Theorem 3.4]{MRS}. The immediate goal of \cite{MRS} was an investigation of postcritically fixed Newton maps, but the following theorem was also proven there for arbitrary postcritically finite maps.
\begin{theorem} \label{Thm:PolesInGraph}
There exists a positive integer $\N$ so that $\Delta_\N$ contains all poles of $N_p$.
\end{theorem}

A \emph{prepole} is defined to be a preimage of $\infty$. As an immediate corollary we see that each prepole is contained in a Newton graph of sufficiently high level (see \cite[Corollary 3.5]{MRS}).

\begin{corollary}[Prepoles in Newton graph]\label{Cor:PrepolesInNewtonGraph}
Let $m\geq 0$ be an integer, and let $\N$ be as in Theorem \ref{Thm:PolesInGraph}. Then every point in $N_p^{-(m+1)}(\infty)$ is a vertex of $\Delta_{m+\N}$.
\end{corollary}

The immediate goal of \cite{MRS} was to give a classification of postcritically fixed Newton maps in terms of abstract Newton graphs.  However, along the way it was largely shown that the Newton graph of postcritically \emph{finite} Newton maps also satisfy the axioms.  The pertinent definitions and theorem are presented here.

\begin{definition}[Abstract channel diagram]
\label{Def_ChannelDiagram} An \emph{abstract channel diagram} of
degree $d\geq 3$ is a finite graph $\Delta$ in $\S^2$ with vertices
$v_\infty,v_1,\dots,v_d$ and edges $e_1,\dots,e_l$ that satisfies the
following properties:
\begin{enumerate}
\item[(1)] $l\leq 2d-2$;
\item[(2)] each edge joins $v_\infty$ to some $v_i$ for $i\in\{1,...,d\}$;
\item[(3)] each $v_i$ is connected to $v_\infty$ by at least one edge;
\item[(4)] \label{necessaryCondition} if $e_i$ and $e_j$ both join $v_\infty$ to $v_k$, then each connected component of
$\S^2\setminus \ol{e_i\cup e_j}$ contains at least one vertex of
$\Delta$.
\end{enumerate}
\end{definition}

It is not difficult to check that the channel diagram $\Delta$ constructed for
a Newton map $N_p$ above satisfies conditions of Definition
\ref{Def_ChannelDiagram}. Indeed by construction, $\Delta$ has at
most $2d-2$ edges and it satisfies (2) and (3). Finally, $\Delta$
satisfies (4), because for any immediate
basin $U_{\xi}$ of $N_p$, every component of $\C\setminus U_\xi$
contains at least one fixed point of $N_p$ \cite[Corollary 5.2]{RS}.

The classification of postcritically fixed Newton maps is given in terms of a combinatorial object called the ``abstract Newton graph" \cite{MRS}. We define the term almost identically except that Condition 3 is relaxed from equality to an inequality (this corresponds to the fact that postcritically finite maps may have critical points that are not eventually fixed).

\begin{definition}[Abstract Newton graph]
\label{Def_NewtonGraph} Let $\Gamma$ be a connected finite
graph in $\S^2$ with vertex set $V(\Gamma)$ and $f:\Gamma\to\Gamma$ a
graph map. The pair $(\Gamma,f)$ is called an \emph{abstract Newton
graph of level $\N_{\Gamma}$} if it satisfies the following conditions:
\begin{enumerate}
\item[(1)] There exists $d_{\Gamma}\geq 3$ and an abstract channel diagram
$\Delta\subsetneq\Gamma$ of degree $d_\Gamma$ such that
 $f$ fixes each vertex and each edge of $\Delta$ (pointwise).

\item[(2)] If $v_\infty, v_1, \dots,v_{d_\Gamma}$ are the vertices of $\Delta$, then
$v_i\in \ol{\Gamma\setminus\Delta}$ if and only if $i\neq \infty$.
 Moreover, there are exactly $\deg_{v_i}(f)-1\geq 1$ edges in $\Delta$ that connect $v_i$ to $v_\infty$ for
 $i\neq \infty$.

\item[(3)] $\sum_{x\in V(\Gamma)} \left(\deg_x f-1\right) \leq 2d_{\Gamma}-2$. 

\item[(4)] $\N_\Gamma$ is the minimal integer so that $f^{\N_\Gamma-1}(v)\in\Delta$ for all $v\in V(\Gamma)$ with $\deg_v f>1$.

\item[(5)] $f^{\N_\Gamma}(\Gamma)\subset\Delta$

\item[(6)] For every $v\in V(\Gamma)$ with $f^{\N_\Gamma-1}(v)\in\Delta$, the number of adjacent edges in $\Gamma$ equals $\deg_v f$ times the number of edges adjacent to $f(v)$.

\item[(7)] The graph $\overline{\Gamma\setminus\Delta}$ is connected.

\item[(8)] \label{Cond_Extension} For every vertex $y\in V(\Gamma)$ and every component $U$ of $\S^2\setminus\Gamma$, the local extension $\tilde{f}$ from Equation (\ref{eqn:localExtension}) is injective on
$\bigcup_{v\in f^{-1}(y)} U_v \cap U\;.$

\end{enumerate}
\end{definition}

It follows from \cite{MRS} and \cite[Corollary 3.2]{DLSS} that if $N_p$ is a postcritically \emph{finite} Newton map, then the pair $(\Delta_{\N},N_p)$ satisfies all conditions of an abstract Newton graph (Definition \ref{Def_NewtonGraph}) for large enough $\N$ (cf. the weaker Theorem 1.5 from \cite{MRS} which only applied to postcritically \emph{fixed} maps)

\begin{theorem}\label{Thm_NewtonGraph}
For every postcritically finite Newton map $N_p$, there exists some minimal level $\N$ so that  $(\Delta_k,N_p)$ is an abstract Newton graph of level $k$ for all $k\geq \N$.
\end{theorem}

Note that the level $\N$ in this theorem is not necessarily the level of the Newton graph chosen in the construction of the extended Newton graph (see Definition \ref{Def_TheNewtonGraphPCFMap}) though it does give a lower bound.

The extended Newton graph that we will associate to a Newton map is a finite graph $\Delta^*_\N$ equipped with a self-map coming from the restriction of $N_p$ (Definition \ref{defn_extendedNewtonGraph}).  This restriction is not a graph map in general since Newton ray edges can contain finitely many preimages of vertices in the Newton graph that are not vertices in $\Delta^*_\N$.  This motivates the following weaker definition where the condition on preimages of vertices has been dropped.

\begin{definition}[Weak graph map]
 A continuous map $f:\Gamma_1 \to \Gamma_2$ is called a \emph{weak graph map} if it is injective on each edge of the graph $\Gamma_1$ and the 
image of each vertex is a vertex.
\end{definition}

\begin{remark}\label{rem_weakGraphMap}
Given a weak graph map $f:\Gamma_1 \to \Gamma_2$, we associate a canonical graph map. Denote by $V_2$ the set of vertices of $\Gamma_2$. Let $\hat{\Gamma}_1$ be the graph with vertex set $f^{-1}(V_2)$ and edges given by the closures of complementary components of $\Gamma_1\setminus f^{-1}(V_2)$. We take $\hat{f}:=f$. Forward and backward invariance of vertices under $\hat{f}$ is immediate.  Since each edge $e$ in $\hat{\Gamma}_1$ is a subset of an edge in $\Gamma_1$ and $f$ is injective on edges in $\Gamma_1$, the restriction $\hat{f}|_e=f|_e$ must also be injective.   It follows that $\hat{f}:\hat{\Gamma}_1\rightarrow\Gamma_2$ is a graph map.
\end{remark}


\section{Hubbard trees from Newton maps} \label{Sec_RenormalizationNewton} 

The Newton graph of the previous section contains all fixed postcritical points of a Newton map as well as all eventually fixed postcritical points. In this section we use Hubbard trees to locally model Newton map dynamics about higher-period postcritical points. 

\subsection{Extended Hubbard trees} \label{Sec_HubbardTrees}

Douady and Hubbard \cite{DH84} showed how to extract from any postcritically finite polynomial a combinatorial invariant called the Hubbard tree, and it was shown that such trees distinguish inequivalent polynomials.  The complete classification of postcritically finite polynomials in terms of Hubbard trees is given in \cite{Poirier}.

A \emph{tree} is a topological space which is uniquely arcwise
connected and homeomorphic to a union of finitely many copies of the
closed unit interval. All trees are assumed to be embedded in $\S^2$.

Let $f$ be a complex polynomial.  Define the
\emph{filled Julia set} $K(f)$ to be the set of $z
\in \C$ so that the forward orbit of $z$ under $f$ is bounded. The \textit{Julia set} $J(f)$ is the boundary
of $K(f)$.

We recall some facts about the dynamics of postcritically finite polynomials; see e.g. \cite{MilnorBook}. 
 For each bounded Fatou component $U_i$, there is exactly one point $x \in U_i$ such that $f^n(x) \in P_f$ for some non-negative integer $n$.  We call $x$ the \emph{center} of $U_i$. Denote by $U_{i+1}$ the Fatou component containing $f(x)$. A classical theorem of B\"ottcher implies that there are holomorphic isomorphisms $\phi_i: (\D,0) \to (U_i,x)$ and $\phi_{i+1}: (\mathbb{D},0) \to (U_{i+1},f(x))$ such that for all $z \in \D$:
\[
\phi_{i+1}(z^{k_i}) = f(\phi_i(z)),
\]
where $k_i$ is the local degree degree of $f$ near $x$.
If $f$ is a postcritically finite polynomial, then the Julia set $J(f)$ is a connected and locally connected compact set \cite{DH84}. Since each Fatou component has locally connected boundary, Caratheodory's theorem implies that the map $\phi_i$ extends continuously to the unit circle. Let $R(t) = \{r \exp(2\pi i t) | 0 \leq r \leq 1 \}$ be the ray of angle $t$ in $\mathbb{D}$. The image $R_i(t) = \phi_i(R(t))$ is called the \emph{ray of angle $t$ in $U_i$}. If $x=\infty$, the ray $R_i(t)$ is called an \emph{external ray}, otherwise it is called \emph{internal ray}.

We now describe the construction of Hubbard trees for a postcritically finite polynomial $f$ following the second chapter of \cite{DH84}.  A Jordan arc $\gamma \subset K(f)$ is called \emph{allowable} if for every Fatou component $U_i$, the set $\phi_i^{-1}(\gamma \cap \overline{U_i})$ is contained in the union of two internal rays of $\overline{\mathbb{D}}$. They show that for every $z,z'$ in $K(f)$ there is a unique
allowable arc joining them \cite[Proposition 2.6]{DH84}. We denote this arc by $[z,z']_{K(f)}$. We
say that a subset $X \subset K(f)$ is \emph{allowably connected} if
for every $z_1,z_2 \in X$ we have $[z_1,z_2]_{K(f)} \subset X$. 
The intersection of a family of allowably connected subsets is allowably connected.
The \emph{allowable hull} $[X]_K$ of $X \subset K(f)$ is defined to be the intersection of all the allowably connected subsets of $K(f)$ containing $X$.  If $X$ is a finite set, then the allowable hull $[X]_K$ is a topological finite tree \cite[Proposition 2.7]{DH84}.

In the following definition (\cite[Definition I.1.9]{Poirier}), $C_f$ denotes the set of critical points.

\begin{definition}[Hubbard tree]
\label{HubbardTree} Let $f$ be a postcritically finite polynomial, and let $M$ be a finite forward invariant set with $C_f\subset M$.  The \textit{Hubbard tree} $H(M)$ is the tree generated by $M$, i.e.\ the allowable hull $[M]_K$.
\end{definition}

Typically $M=P_f$ in the literature.  We will often wish to include other points as discussed below.

These Hubbard trees (including those with additional marked points) are axiomatized as \emph{abstract Hubbard trees} in Section II.4 of \cite{Poirier} (see also \cite{Poi}).   Poirier assigns a degree to each Hubbard tree in terms of local degree of the tree dynamics (he always assumes that the degree is greater than one).  Under a natural partial ordering on abstract Hubbard trees, Poirier shows that there is a unique minimal abstract Hubbard tree that is in fact the tree generated by the orbit of $C_f$ \cite[Proposition II.4.5]{Poirier}.   An equivalence relation on abstract Hubbard trees is given in \cite[Definition II.4.3]{Poirier}, where two trees are equivalent if they are homeomorphic and the dynamics are respected (among other things, this means the dynamics on vertices are conjugate).  

We now state two essential theorems, begining with the crucial realization theorem for Hubbard trees \cite[Theorem II.4.7]{Poirier}.

\begin{theorem} [Realization of abstract Hubbard trees]\label{Thm_RealizeAbstrHubTree}
For every abstract Hubbard tree $H$, there exists a unique (up to affine conjugacy) postcritically finite polynomial $f$ such that $H(M)$ is equivalent to $H$ where $M\supset C_f$ is a finite forward-invariant set.
\end{theorem}

Next we give the classification theorem which is essentially a consequence of the construction of Hubbard trees, the realization theorem, and minimality \cite[Theorem II.4.8]{Poirier}. 

\begin{theorem}[Classification of postcritically finite polynomials]\label{Thm_ExtendedHubbardTreesBijectiveCorr}
The set of affine conjugacy classes of postcritically finite polynomials of degree at least two is in bijective correspondence with the set of equivalence classes of minimal abstract Hubbard trees. 
\end{theorem}

We must now give an analogous exposition for polynomials where all cycles up to a certain length are marked.  The reason is that a Newton map may have a non-fixed polynomial-like map, and we would like to use a ray to connect a repelling periodic point in the Hubbard tree of this polynomial-like map to the Newton graph.  It is also necessary at times to mark cycles of longer length so that we may maintain combinatorial control of the free critical points of the Newton map that map into repelling cycles of filled Julia sets.  

The set of marked points for the polynomial $f$ including cycles of length $n$ or less is denoted
\[M_n=C_f\cup P_f\cup\bigcup_{i=0}^n\{z\in K(f)|f^{i}(z)=z\}.\]

\begin{definition}[Extended Hubbard tree]\label{def_ExtendedHubTree} An \emph{extended Hubbard tree} is a Hubbard tree of the form $H(M_n)$ where $n\geq 1$.  We say that $H(M_n)$ has \emph{cycle type} $n$.
\end{definition}

\begin{remark}  If $f:\Cc\to\Cc$ is a degree one holomorphic map with a unique finite repelling fixed point $z_0\in\mathbb{C}$, the extended Hubbard tree $H(M_n)$ consists of the point $z_0$ equipped with the identity map for all $n$ (see Definition \ref{HubbardTree}).  If an extended Hubbard tree consists of a single point, it is said to be \textit{degenerate}. The only polynomials with degenerate Hubbard trees are degree one.
\end{remark}

As mentioned, the definition of abstract Hubbard tree (Section II.4 \cite{Poirier}) allows for marked points beyond the postcritical set. For an abstract extended Hubbard tree $H$, let $deg(H)$ be the degree of the polynomial associated to $H$ by Theorem \ref{Thm_ExtendedHubbardTreesBijectiveCorr}.
\begin{definition}
An \emph{abstract extended Hubbard tree} is an abstract Hubbard tree $H$ whose vertex set has $(deg(H))^k$ points that are fixed by the graph map for each $1\leq k\leq n$. Such a tree is said to have cycle type $n$.
\end{definition}

The partial order on abstract Hubbard trees defined by \cite[Definition II.4.2]{Poirier} induces an order on abstract extended Hubbard trees of fixed degree and fixed cycle type $n$.  In analogy to \cite[Proposition II.4.5]{Poirier}, we conclude that there is a unique minimal abstract extended Hubbard tree under this partial order, namely the tree generated by the points in cycles of length $n$ or less.  By convention, the minimal degree one abstract extended Hubbard tree is the degenerate Hubbard tree.

Since each extended abstract Hubbard tree is in fact an abstract Hubbard tree (except in degree one where realization is evident anyway), we may apply Theorem \ref{Thm_RealizeAbstrHubTree} (Poirier's realization theorem) to abstract extended Hubbard trees.

\subsection{Polynomial-like maps and renormalization} \label{Sec_PolynomialLike}

Polynomial-like maps were introduced by Douady and Hubbard \cite{DH85} and have played an important role in complex dynamics ever since.  They will be used in Section \ref{Sec:RenormOfNewtonMaps} to model the dynamics close to critical points whose orbit does not intersect the channel diagram.

\begin{definition}\label{Def_PolyLikeMaps} A \emph{polynomial-like} map of degree $d\geq 1$ is a triple
$(f,U,V)$ where $U,\,V$ are open topological disks in $\Cc$, 
the set $\overline{U}$ is a compact subset of $V$, and $f: U \to
V$ is a proper holomorphic map such that every point in $V$ has $d$
preimages in $U$ when counted with multiplicities.
\end{definition}
\begin{remark}\label{rmk:DegenerateHubTrees}
The above definition differs slightly from the typical one found in the literature, as we allow that $d=1$.  Such a map is called a \emph{degenerate polynomial-like map}.  The following two theorems are stated in their original sources for $d\geq 2$, but we include the $d=1$ case without justification, as the proof in this case is trivial.
\end{remark}

\begin{definition}[Filled Julia set] Let $f: U \to V$ be a polynomial-like map. The \emph{filled
Julia set} of $f$ is the set of points in $U$ that never
leave $V$ under iteration of $f$, i.e.
\begin{center}
$\displaystyle K(f,U,V) = \bigcap_{n=1}^\infty f^{-n}(V)$.
\end{center}
When context permits, the set will more simply be denoted $K(f)$. As with polynomials, we define the \emph{Julia set} as
$J(f)=\partial{K(f)}$.
\end{definition}

The simplest example of a polynomial-like map comes from restricting an actual polynomial: let $p$ be a polynomial of degree $d\geq 2$, let $V=\{z\in \mathbb{C}: |z|<R \}$ for sufficiently large $R$ and $U=f^{-1}(V)$. Then $p: U \to V$ is a polynomial-like mapping of degree $d$.

Two polynomial-like maps $f$ and $g$ are \emph{hybrid equivalent} if
there is a quasiconformal conjugacy $\psi$ between $f$ and $g$ that is
defined on a neighborhood of their respective filled Julia sets
so that $\bar{\partial}\psi = 0$ on $K(f)$.

The crucial relation between polynomial-like maps and polynomials is
explained in the following theorem, due to Douady and Hubbard \cite{DH85}. 

\begin{theorem}[The straightening theorem]\label{thm:straightening}
Let $f: U' \to U$ be a polynomial-like map of degree $d$.
Then $f$ is hybrid equivalent to a polynomial $P$ of degree $d$.
Moreover, if $K(f)$ is connected and $d\geq 2$, then $P$ is unique up to affine
conjugation.
\end{theorem}

As an immediate consequence of the first part of the theorem, it follows that $K(f)$ is connected if and only if $K(f)$ contains the critical points of $f$. Now we define the notion of renormalization of rational functions. Let $R$ be a rational function of degree $d$ and let $z_0$ be a critical or postcritical
point of $R$.  

\begin{definition}
\label{Def_Renorm} $R^n$ is called \emph{renormalizable about} $z_0$
if there exist open disks $U,V \subset \C$ satisfying the following
conditions:
\begin{enumerate}
\item $z_0\in U$.
\item $(R^n,U,V)$ is a polynomial-like map with critical points contained in the filled Julia set.
\end{enumerate}
Such a triple $\rho = (R^n,U,V)$ is called a \emph{renormalization}.
\end{definition}

\begin{remark}
The usual definition of renormalization requires that $(R^n,U,V)$ is polynomial-like in the standard sense, i.e. of degree at least 2. We expand the definition to include degree 1  maps for critical points that eventually land on repelling periodic points. The common theme is that every renormalization captures the dynamics of some critical orbit(s) that are eventually periodic, at least in a combinatorial sense (with respect to the partition induced by the Newton graph of some high level). Renormalizations of degree 1 that do not capture postcritical points are of no interest for us.
\end{remark}

In Section \ref{Sec_HubbardTrees} the notion of extended Hubbard trees for a given postcritically finite polynomials was introduced. Note that the same construction applies to polynomial-like maps $f:U' \to U$ with connected filled Julia set. We use this in the following.

\subsection{Renormalization of Newton maps}\label{Sec:RenormOfNewtonMaps}

Preliminaries aside, we now turn to the main goal of this section of giving a local combinatorial description of the critical and postcritical points of a Newton map that are not eventually fixed. The description relies on the fact that all periodic points can be described in terms of renormalization using \cite{DLSS}. Those results were made for the more general class of ``attracting-critically-finite" Newton maps, but for simplicity we specialize statements to postcritically finite Newton maps.  Another slight difference is that in the present work, polynomial-like maps are permitted to have degree one.

The Hubbard trees we use to model filled Julia sets will contain all fixed points of the first return of the polynomial-like map, so it is important for our classification that the polynomial-like maps are not iterates of polynomial-like maps (iterates have inequivalent Hubbard trees). A renormalization $\rho=(R^n, U, V)$ is said to be of \emph{lowest period} if $n$ is the minimal integer so that $R^n(K(\rho))=K(\rho)$, recalling that $K(\rho)$ denotes the filled Julia set of the renormalization. 

 Denote by $Q$ the set of critical and postcritical points of $N_p$ that are not eventually fixed. The next two propositions are Propositions 5.1 and 5.2 of \cite{DLSS}. 

\begin{proposition}[Lowest period renormalization]
\label{Prop:NonRepellingRenormalizableMinimalPeriod} Let $N_p$ be a postcritically finite Newton map. If $q\in Q$ is periodic then it is lowest period renormalizable.
\end{proposition}

The filled Julia set of the renormalization at periodic $q\in Q$ is denoted by $K(q)$. Furthermore, if $q\in Q$ is not in the filled Julia set of any renormalization from Proposition \ref{Prop:NonRepellingRenormalizableMinimalPeriod}, we define $K(q)$ to be the component of $N_p^{-i}(K(q'))$ where $i$ is minimal so that $N_p^{i}(K(q))=K(q')$ for some $q'$ in the filled Julia set of a renormalization.

\begin{proposition}[Separability of filled Julia sets]
\label{Prop:SeparabilityOfFibers} For all $q\in Q$, the set $K(q)$ does not intersect the Newton graph of any level. Furthermore, there is a level of the Newton graph so that for all $q,q'\in Q$, either $K(q)$ and $K(q')$ are in different complementary components of the Newton graph or $K(q)=K(q')$.
\end{proposition}

\begin{remark}[Hubbard trees from renormalizations]\label{Rmk:HubFromRenorm}
Each polynomial-like map constructed in Proposition \ref{Prop:NonRepellingRenormalizableMinimalPeriod} can be modeled using a Hubbard tree. Specifically, for $q\in Q$ with periodic $K(q)$, denote by $H(q)$ the extended Hubbard tree of the lowest period renormalization at $q$ whose cycle type is minimal so that all postcritical points of $N_p$ in $K(q)$ are vertices. For $q$ with $K(q)$ not periodic, let $i>0$ be minimal so that $N_p^{i}(q)$ lies in a periodic Hubbard tree $H_j$ (of the type just constructed); we define $H(q)$ to be the component of $N_p^{-i}(H_j)$ that contains $q$.
\end{remark}

\begin{corollary}[Separability of Hubbard trees] \label{Rem_ExtHubbardTrees}  For all $q\in Q$, the tree $H(q)$ does not intersect the Newton graph of any level. Furthermore, there is a level of the Newton graph so that any two trees $H(q)$ and $H(q')$ with $q,q'\in Q$ are either equal or lie in different complementary components of the Newton graph.
\end{corollary}

\begin{proof}
For each $q$, the containment $H(q)\subset K(q)$ holds. Proposition \ref{Prop:SeparabilityOfFibers} asserts that the separation property holds for the sets $K(q)$, and so it must hold for the trees $H(q)$. 
\end{proof}


\section{Newton rays from Newton maps} \label{Sec_Newtrays}
We now construct Newton ray edges, which will connect the repelling fixed points of the extended Hubbard trees constructed in the previous chapter to the Newton graph (see Theorems \ref{Lem_RepPtsAreLandingPts} and \ref{Lem_RepPtsFixedRays}).  There will be two types of Newton ray: one type is very simply a periodic ray in the immediate basin of a root, and the other type will be defined as subsets of ``bubble rays", chains of Fatou components that have been used in the literature in several situations \cite{YZ,Ro98,Lu}.

Let $\N$ be a level of the Newton graph $\Delta_\N$ so that every pole of $N_p$ and every critical point of $N_p$ that eventually lands in the channel diagram $\Delta$ is contained in $\Delta_{\N}$. The existence of such a level is guaranteed by \cite[Corollary 3.5]{MRS}. It is very possible that a higher level will be taken in the actual construction of the combinatorial invariant for $N_p$ (see Definition \ref{Def_TheNewtonGraphPCFMap}). Nevertheless, the results of this section hold in either case.

\begin{definition} \label{Def_NewtonRay}
A \emph{Newton ray} is a simple path $\mathcal{R}$ beginning at a vertex $v\in\Delta_\N$ and terminating at $z_0$ in the Julia set so that $\mathcal{R}\cap\Delta_\N=\{v\}$. We say that $\mathcal{R}$ \emph{lands} at $z_0$. The Newton ray $\mathcal{R}$ is said to be \emph{periodic} if there exists an integer $m \geq 1$ such that $N^m_p(\mathcal{R}) = \mathcal{R} \cup \mathcal{E}$, where $\mathcal{E} \subset \Delta_\N$ is a (possibly empty) subgraph in $\Delta_\N$ that depends on $\mathcal{R}$. The smallest such $m$ is \emph{the period} of $\mathcal{R}$.
\end{definition}

In the case that $z_0$ does not lie in the boundary of an immediate basin of a root, the desired periodic Newton ray landing at $z_0$ will be extracted from a bubble ray, which we now define. For example, see Figure \ref{Fig_NewtonDeg4BubbleRays} which displays period two Newton rays (indicated by yellow edges) together with a schematic of the associated bubble ray.

\begin{definition}
A \emph{bubble} of $N_p$ is a Fatou component in the basin of attraction of one of the fixed critical points of $N_p$.  The set of bubbles for $N_p$ is denoted $\mathfrak{B}(N_p)$.

\end{definition}

The \emph{center} of a bubble $B$ is the unique point of $B$ which eventually maps to a fixed critical point under $N_p$.  Two distinct bubbles with intersecting closures are said to be \emph{adjacent}. 

\begin{definition}[Bubble rays and chains]  A set of the form
\[
\overline{\bigcup_{i=n}^m B_i}
\]
where $B_0, B_1 , \ldots$ are distinct bubbles, $B_i$ is adjacent to $B_{i-1}$ for all $i \geq 1$, and $N_p(B_0)=B_0$ is called a
\begin{itemize}
\item \emph{bubble ray} if $n=0, m=\infty$,
\item \emph{finite bubble ray} if $n=0, m<\infty$,
\item \emph{finite bubble chain} if $n>0, m<\infty$.
\end{itemize}
\end{definition}

We now impose a tree structure on bubbles that is respected by $N_p$ (see Lemma \ref{Lem_BubblePredecessorPreserved}). Choose a maximal subtree $T_0\subset\Delta_0=\Delta$.  Inductively define $T_i\subset\Delta_i$ to be a maximal subtree of $N_p^{-1}(T_{i-1})\cap\Delta_{i}$ so that additionally $T_{i}\supset T_{i-1}$.  By construction, \[N_p(T_i)\subset T_{i-1}.\] Let $\mathfrak{B}_i$ be the set of bubbles that intersect $T_i$. Clearly \[\mathfrak{B}_0\subset\mathfrak{B}_1\subset\mathfrak{B}_2\subset\dots.\] It is a consequence of Corollary \ref{Cor:PrepolesInNewtonGraph} that 
\[\bigcup_{i>0}\mathfrak{B}_i=\mathfrak{B}(N_p).\]
\begin{definition}[Bubble parents] The \emph{parent function} $\mathscr{P}_i:\mathfrak{B}_i\to\mathfrak{B}_i$ is given by
\[\mathscr{P}_i(B)=
\begin{cases}
B, \text{  when } N_p(B)=B\\
B', \text{  otherwise}
\end{cases}
\]
where $B'$ is the unique bubble adjacent to $B$ intersecting the unbounded component of $T_i\setminus B$.
\end{definition}

\begin{remark}[Choice of bubble level]\label{Rmk:BubbleLevel}
It is easily seen that $\mathscr{P}_i=\mathscr{P}_j$ on $\mathfrak{B}_i\cap\mathfrak{B}_j$. Thus parent function values are not changed when $i$ is increased, so in practice we can often ignore the subscript. We always assume that $i$ is large enough so that $T_{i}$ contains all poles of $N_p$ and all critical points of $N_p$ that are eventually fixed. Consequently, all such poles and critical points are contained in the closure of some bubble in $\mathfrak{B}_{i}$.
\end{remark}

There is a finite subset of exceptional ``bad bubbles" in $\mathfrak{B}_{i}$ where the conclusion of Lemma \ref{Lem_BubblePredecessorPreserved} may not hold. Let $\mathfrak{B}^{bad}_0$ be the finite set of bubbles whose closures intersect a pole or an eventually-fixed critical point of $N_p$. Then 
\[\mathfrak{B}^{bad}:=\bigcup_{j\geq 0}\mathscr{P}_i^{ j}(\mathfrak{B}^{bad}_0) \] 
This definition is independent of $i$ due to Remark \ref{Rmk:BubbleLevel}.

\begin{lemma} \label{Lem_BubblePredecessorPreserved}
For all bubbles $B\in\mathfrak{B}_i\setminus\mathfrak{B}^{bad}$,
\[N_p(\mathscr{P}_i(B))=\mathscr{P}_i(N_p(B)).\]

\end{lemma}

\begin{proof}
There is an oriented simple path $\gamma$ in $T_i$ that begins at the center of $\mathscr{P}_i(B)$, then intersects the center of $B$, and then terminates at an end of $T_i$. Then because $\gamma$ does not pass through $\mathfrak{B}^{bad}$, it does not pass through any critical point of $N_p$. It follows that $N_p(\gamma)$ is a simple arc that begins at $N_p(\mathscr{P}_i(B))$, then intersects the center of $N_p(B)$, and then terminates at an end of $T_{i-1}$. Since $\gamma$ passed through no poles, $N_p(\gamma)$ does not pass through $\infty$. It follows that $N_p(\mathscr{P}_i(B))$ separates $N_p(B)$ from $\infty$ in $T_{i-1}$. But since $N_p(\mathscr{P}_i(B))$ is adjacent to $N_p(B)$ the equation follows.
\end{proof}

We will now associate a bubble ray to each bubble, and show that $N_p$ respects this assignment. This will be crucial in the proof of Theorem \ref{Lem_RepPtsAreLandingPts}. 

\begin{definition}[Bubble rays from bubbles] Let $B\in\mathfrak{B}_i$ be a bubble. The \emph{associated finite bubble ray} is given by
\[\widehat{B}=\overline{\bigcup_{k\geq 0}\mathscr{P}_i^k(B)}.\] The total number of bubbles in $\widehat{B}$ is denoted $|\widehat{B}|$.
\end{definition}

We need the following fact that $N_p$ maps the ray associated to a bubble to another such ray.

\begin{lemma}\label{lem:AssociatedRaysMaptoRays}
Let $B\in\mathfrak{B}_i$ be a bubble. Then 
\[N_p(\widehat{B})=\widehat{N_p(B)}\cup\mathcal{E}\]
where $\mathcal{E}$ is a finite set of bubbles that intersect $\Delta_\N$ and depend on $\widehat{B}$.
\end{lemma}
\begin{proof}
If $B$ intersects $\Delta_\N$, then each bubble in $\widehat{B}$ intersects $\Delta_\N$. By the forward invariance of $\Delta_\N$, it follows that each bubble in $\widehat{N_p(B)}$ intersects $\Delta_\N$, and the result is immediate.  

The other case is that $B$ does not intersect $\Delta_\N$. Since $\Delta_\N$ intersects each bubble in $\mathfrak{B}^{bad}$, it follows that $B\notin\mathfrak{B}^{bad}$. Let $j$ be the smallest number so that $\mathscr{P}^{j}(B)\cap \Delta_\N\neq\emptyset$, where of necessity $j\geq1$.  By Lemma \ref{Lem_BubblePredecessorPreserved},
\[N_p(\bigcup_{k=0}^{j-1}\mathscr{P}^k(B))=\bigcup_{k=0}^{j-1}\mathscr{P}^k(N_p(B))\subset\widehat{N_p(B)}\]
Then by the forward invariance of $\Delta_\N$ under $N_p$, we have that
\[\mathcal{E}:=N_p\big(\bigcup_{k>j-1}\mathscr{P}^k(B)\big)\]
is a collection of bubbles that intersect $\Delta_\N$.
\end{proof}

\begin{theorem}[Newton rays for repelling cycles] \label{Lem_RepPtsAreLandingPts}
Let $\omega$ be a repelling periodic point of period $m > 1$ of $N_p$. Then there exists a periodic Newton ray $\mathcal{R}$ landing at $\omega$ whose period is divisible by $m$.
\end{theorem}

\begin{proof} 
Suppose first that $\omega$ is in the boundary of an immediate basin of a root. Since $N_p$ is postcritically finite, its Julia set is locally connected. Thus the boundary of the immediate basin is locally connected and so there is a periodic internal ray $\mathcal{R}$ that connects $\omega$ to the root. Evidently $\mathcal{R}$ has period divisible by $m$.

The rest of the proof concerns the case when $\omega$ is not in the boundary of an immediate basin of a root. Let $Y$ be a neighborhood of $\omega$ that is a disk in linearizing coordinates. By passing to a smaller linearizing neighborhood, it may be assumed that $\infty\notin Y$.   Let $h:Y\to\Cc$ be the branch of $N_p^{-m}$ fixing $\omega$, from which it follows that $h(Y)\subset Y$.

We first show the existence of a bubble in $Y$ that does not intersect $\Delta_\N$. It is known from Proposition \ref{prop:Jconnected} that the Julia set is connected. Since both $\infty$ and $\omega$ are in the Julia set, there is some point $z_0$ in the interior of $Y$ that is also in the Julia set. Let $U_1$ be a non-fixed preimage of a fixed Fatou component.  Let $z_1\in\partial U_1$ be a pole, which is evidently in the Julia set. Recall the standard fact that the preimages of $z_1$ are dense in the Julia set. Thus there is a sequence $\{z_k\}$ of distinct points in the Julia set so that $z_k\to z_0$ and $N_p (z_k)=z_{k-1}$   for $k\geq 2$. Note that each $z_k$  is in the boundary of some Fatou component $U_k$  that satisfies $N_p^{k-1} (U_k)=U_1$ where all the Fatou components $U_k$ are distinct. Each $U_k$ is evidently a bubble. 

For each $\epsilon >0$ there are only finitely many Fatou components with spherical diameter larger then $\epsilon$ \cite[Lemma 19.4 ff.]{MilnorBook}. Thus $diam(U_k )\to 0$ as $k\to\infty$. It follows from the triangle inequality that $diam(U_k\cup \{z_0\} )\to 0$. Thus there is some choice of $k$ so that $U_k\subset Y$. Since such a $U_k$ can be chosen with arbitrarily small diameter, we may assume that it is not one of the finitely many bubbles that intersect $\Delta_\N$. 

We have thus proven the existence of a bubble $B_0$ in $Y$ that does not intersect $\Delta_\N$. Let $B_i:=h^{i}(B_0)$. We first give a linear upper bound for the total number of bubbles in the bubble ray $\widehat{B}_i$. Let $M$ be the number of bubbles in:
\[\bigcup_{k\geq 0}\mathscr{P}^{k}(\bigcup_{1\leq j \leq d} N_p^{-m}(A_j))\]
where the reader is reminded that $A_1,...,A_d$ are the bubbles fixed by $N_p$. It follows from Lemma \ref{lem:AssociatedRaysMaptoRays} that $N_p^m(\widehat{B}_i)=\widehat{B}_{i-1}\cup\mathcal{E}$ where $\mathcal{E}$ is a finite union of bubbles that intersect $\Delta_\N$.
Then for all $i$,
\[|\widehat{B}_i|-|\widehat{B}_{i-1}|<M\]
so by induction,
\[|\widehat{B}_i|<|\widehat{B}_{0}|+M\cdot i\] from which it follows that
\begin{equation}\label{equation:bubbleUpperBound}
\frac{1}{|\widehat{B}_i|}>\frac{1}{|\widehat{B}_{0}|+M\cdot i}.
\end{equation}

Now let $\widehat{B}_i^Y$ be the largest bubble chain in $\widehat{B}_i$ that contains $B_i$ and consists only of bubbles that are a subset of $Y$.

\begin{lemma}\label{lemma:BubbleRaysOverlap}
There exist $i,j\geq 0$ so that $i>j$ and $\widehat{B}_i^Y\cap\widehat{B}_j^Y$ contains at least one bubble.
\end{lemma}
\begin{proof}[Proof of Lemma \ref{lemma:BubbleRaysOverlap}]
Suppose the contrary, namely that if $\widehat{B}_i^Y\cap\widehat{B}_j^Y$ contains a bubble then $i=j$. For each $i$, let $\textrm{Max}(\widehat{B}_i^Y)$ be a bubble in the bubble chain $\widehat{B}_i^Y$ having maximal spherical diameter, and let $\delta_i:=diam(\textrm{Max}(\widehat{B}_i^Y))$. By hypothesis, $\textrm{Max}(\widehat{B}_i^Y)=\textrm{Max}(\widehat{B}_j^Y)$ implies that $i=j$. It is a consequence of subhyperbolicity that there are only finitely many Fatou components with diameter in the spherical metric greater than an arbitrarily fixed constant \cite[Lemma 19.4 ff.]{MilnorBook}. Thus $\delta_i\to 0$.  

Since $\omega$ is periodic and assumed to not be in the boundary of any immediate basin of a root, it is clear that $\omega$ is not in the boundary of any bubble. Let $\mathcal{U}_Y$ be the union of all bubbles that intersect $\partial Y$. Then $\mathcal{F}_Y=\overline{Y}\setminus\mathcal{U}_Y$ is a closed set, and by the finiteness result mentioned above, the spherical distance $d(\omega,\partial \mathcal{F}_Y)$ is positive. Since $B_i\to \omega$, it follows that $d(\overline{B}_i,\partial \mathcal{F}_Y)\to d(\omega, \partial \mathcal{F}_Y)>0$.  The closure of $\widehat{B}_i^Y$ intersects $\partial \mathcal{F}_Y$ by definition, and since $\delta_i\to 0$, the number of bubbles in $\widehat{B}_i^Y$ must tend to $\infty$. 

Let $A$ be the fundamental annulus bounded by $\partial Y$ and $h(\partial Y)$. By a similar argument there exists $K\in\mathbb{Z}^+$ so that the number of bubbles in $\widehat{B}_i^Y\cap A$ is greater than $2M$ for all $i>K$. Consequently,
\begin{equation}\label{equation:bubbleLowerBound}
|\widehat{B}_i^Y|\geq 2M(i-K)
\end{equation}
for each $i>K$. Then
\[\lim_{i\to\infty}\frac{|\widehat{B}_i^Y|}{|\widehat{B}_i|} {\geq} \lim_{i\to\infty}\frac{|\widehat{B}_i^Y|}{|\widehat{B}_0|+M\cdot i} {\geq} \lim_{i\to\infty}\frac{2M(i-K)}{|\widehat{B}_0|+M\cdot i}=2\] where the first inequality follows from multiplying both sides of Inequality (\ref{equation:bubbleUpperBound}) by $|\widehat{B}_i^Y|$ and passing to the limit, and the second follows from Inequality (\ref{equation:bubbleLowerBound}) in a similar way.
This contradicts the fact that $\frac{|\widehat{B}_i^Y|}{|\widehat{B}_i|}\leq 1$ 
\end{proof}

The proof of Theorem \ref{Lem_RepPtsAreLandingPts} is now resumed. Let $R_0\subset\widehat{B}_i^Y\cap\widehat{B}_j^Y$ be the bubble guaranteed by Lemma \ref{lemma:BubbleRaysOverlap}. Then the following is a finite bubble chain connecting $B_j$ and $R_0$,
\[(\widehat{B_j}\setminus\widehat{R_0})\cup R_0\]
and so its image under $h^{i-j}$ is a finite bubble chain connecting $B_i=h^{i-j}(B_j)$ to $h^{i-j}(R_0)$. By Lemma \ref{Lem_BubblePredecessorPreserved}, it follows that $R_1:=h^{i-j}(R_0)$ is a bubble in $\widehat{B}_i$. 

Let $\lambda$ be a simple path in $\widehat{R}_1$ connecting the center of $R_0$ to the center of $R_1$, where $\lambda$ is further assumed to be a path in the tree used to define the parent function. The set
\[\lambda' = \bigcup_{k=1}^{\infty}h^{k(j-i)}(\lambda)\cup\{\omega\}\]
forms an arc that evidently satisfies $N_p^{m(j-i)}(\lambda')\supset\lambda'$. It is simple because of Lemma \ref{lem:AssociatedRaysMaptoRays}, but it does not necessarily intersect the Newton graph.
Thus we let $\ell\geq 0$ be the smallest multiple of $j-i$ so that $N_p^{m\cdot\ell}(\lambda')$ intersects $\Delta_\N$, and define the ray
\[\mathcal{R}=\overline{N_p^{m\cdot \ell}(\lambda')\setminus\Delta_\N}.\]

We now show that $\mathcal{R}$ is the desired periodic Newton ray. Since $\lambda'$ was a path in the tree used to define the parent function, Lemma \ref{Lem_BubblePredecessorPreserved} implies that $\mathcal{R}$ is also a path in the tree and hence simple.
By construction, $\mathcal{R}$ lands at $\omega$ and begins at a single vertex of $\Delta_\N$. The periodicity of $\mathcal{R}$ follows from Lemma \ref{lem:AssociatedRaysMaptoRays}.
\end{proof}

For the rest of the section, let $H$ be a Hubbard tree of period $m$ as constructed in Remark \ref{Rmk:HubFromRenorm}, and let $\omega\in H$ satisfy $N_p^{m}(\omega)=\omega$.

\begin{remark} \label{Rem_RaysDisjointFromHubbardTrees}
For any Newton ray $\mathcal{R}$ landing at $\omega$ and for any other Hubbard tree $H'$,
\[
(\mathcal{R}\setminus \{\omega\}) \cap H' = \emptyset.
\]
In the case when $\omega$ lies in the boundary of the immediate basin of a root, this is because $\mathcal{R}\setminus \{\omega\}$ is a subset of the immediate basin of a root, and such immediate basins do not intersect any of the Hubbard trees by Proposition \ref{Prop:SeparabilityOfFibers}. Otherwise, suppose $\omega$ does not lie in the boundary of the immediate basin of a root. By construction any point $x \in \mathcal{R}\setminus \{\omega\}$ is eventually mapped into $\Delta_\N$ by $N_p$, while the orbit of $H'$ under $N_p$ is disjoint from $\Delta_\N$.
\end{remark}

We now produce a ``rightmost" Newton ray in order to prove Theorem \ref{Lem_RepPtsFixedRays} which asserts the existence of a periodic Newton ray of minimal possible period. We will not need to order rays that land at degenerate Hubbard trees, so we assume that $H$ is not degenerate.  Arbitrarily choose an edge in $H$ with $\omega$ as an endpoint.  Denote this edge by $E_\omega$.

Fix the orientation of $\S^2$ to be the counterclockwise orientation for the rest of this paper. 
\begin{definition}[Newton ray order] \label{Def_RightEnvelope}
Let $\mathcal{R}', \, \mathcal{R}''$ be Newton rays landing at $\omega$ and let $E_w$ be an edge in $H$ with endpoint $\omega$.  The Newton rays are said to \emph{not cross-intersect} if they satisfy the following property: if $l$ is a curve disjoint from $\mathcal{R}',\, \mathcal{R}''$ and connecting the endpoints of $\mathcal{R}''$ and $E_w$ different from $\omega$, then $\mathcal{R}'$ intersects only one complementary component of $\Cc \setminus \left(E_{\omega} \cup l \cup \mathcal{R}'' \right)$.  Assume that $\mathcal{R}'$ and $\mathcal{R}''$ don't cross-intersect.  Let $Y$ be a neighborhood of $\omega$ such that for some branch $h=N^{-m}_p$,  we have $h(Y) \subset Y$.  We say that $\mathcal{R}' \succeq \mathcal{R}''$ if for any such neighborhood $Y$, the cyclic order around $\omega$ in $Y$ is $\mathcal{R}', \,\mathcal{R}'', \, E_w$.
\end{definition}

\begin{remark} Note that for any other such neighborhood $Y' \subset Y$, the cyclic order of $\mathcal{R}'', \,\mathcal{R}', \, E_w$ in $Y'$ is the same as in $Y$.  Hence the relation $\succeq$ is well defined and doesn't depend on the choice of the neighborhood $Y$.
\end{remark}

\begin{lemma}
Let $\mathcal{R}_1, \, \mathcal{R}_2$ be periodic (possibly cross-intersecting) Newton rays that land at a repelling fixed point  $\omega$ of $N_p^m$ in a period $m$ Hubbard tree $H$.  Then there is a periodic Newton ray $\mathcal{R} = RE(\mathcal{R}_1,\mathcal{R}_2)$ that satisfies the following properties:
\begin{itemize}
\item $\mathcal{R}$ ends at $\omega$.
\item $\mathcal{R}$ doesn't cross-intersect either $\mathcal{R}_1$ or $\mathcal{R}_2$.
\item $\mathcal{R} \succeq \mathcal{R}_1, \, \mathcal{R} \succeq \mathcal{R}_2$.
\end{itemize}
\end{lemma}

\begin{remark}
Such a ray $RE(\mathcal{R}_1,\mathcal{R}_2)$ is said to be \emph{the right envelope} of Newton rays $\mathcal{R}_1, \mathcal{R}_2$ (see Figure \ref{Fig_RightEnvelope}).
\end{remark}

\begin{proof}
If $\omega$ lies in the boundary of an immediate basin of a root, there is the possibility that one of $\mathcal{R}_1$ and $\mathcal{R}_2$ is an internal ray, which we take to be $\mathcal{R}_1$ without loss of generality. If $\mathcal{R}_1=\mathcal{R}_2$, we take take $\mathcal{R}:=\mathcal{R}_1$. Otherwise $\mathcal{R}_1$ and $\mathcal{R}_2$ intersect only at their endpoints, so if $\mathcal{R}_1 \succeq \mathcal{R}_2$ then $\mathcal{R}:=\mathcal{R}_1$ or if $\mathcal{R}_2 \succeq \mathcal{R}_1$ then $\mathcal{R}:=\mathcal{R}_2$. For the rest of the proof we assume that neither of $\mathcal{R}_1$ and $\mathcal{R}_2$ is an internal ray.

Let $Y$ be a neighborhood of $\omega$ such that for some branch $h=N^{-m}_p, \, h(Y) \subset Y, \, E_w\cap Y \subsetneqq E_w$ and 
$\partial{Y} \cap \mathcal{R}_1 =\{v_1\}, \quad \partial{Y} \cap \mathcal{R}_2 =\{v_2\},$
where $v_1, \, v_2$ are iterated preimages of vertices of $\Delta_\N$ under $N_p$. Let $Y_1,Y_2$ be the connected components of 
$Y \setminus \left(\mathcal{R}_1 \cup \mathcal{R}_2 \cup E_w \right)$
so that 
$E_w\cap Y \subset \partial{Y_1} \cap \partial{Y_2}$ and the cyclic order around $\omega$ is $Y_2, \,E_w, \, Y_1$. We define what will be called the right envelope of $\mathcal{R}_1$ and $\mathcal{R}_2$ in $Y$ by
\[
RE_Y(\mathcal{R}_1,\mathcal{R}_2) := \partial{Y_1} \setminus \left(E_w \cup \partial{Y}\right). 
\]
It follows from the construction that 
\[
RE_Y(\mathcal{R}_1,\mathcal{R}_2) \subset (\mathcal{R}_1 \cup \mathcal{R}_2)\cap Y.
\]
Let $i$ be the smallest integer $i>0$ such that $N^{im}_p(RE_Y(\mathcal{R}_1,\mathcal{R}_2)) \cap \Delta_\N \neq \emptyset$. 
Define
\[
\mathcal{R} = RE(\mathcal{R}_1,\mathcal{R}_2) := \overline{N^{im}_p(RE_Y(\mathcal{R}_1,\mathcal{R}_2))\setminus\Delta_\N}.
\]
It is evident that $\mathcal{R}$ is a Newton ray that lands at $\omega$ and satisfies the desired properties.

\end{proof}

\begin{figure}[hbt]
\begin{center}
\setlength{\unitlength}{1cm}
\begin{picture}(11,7)
\includegraphics[width=12cm]{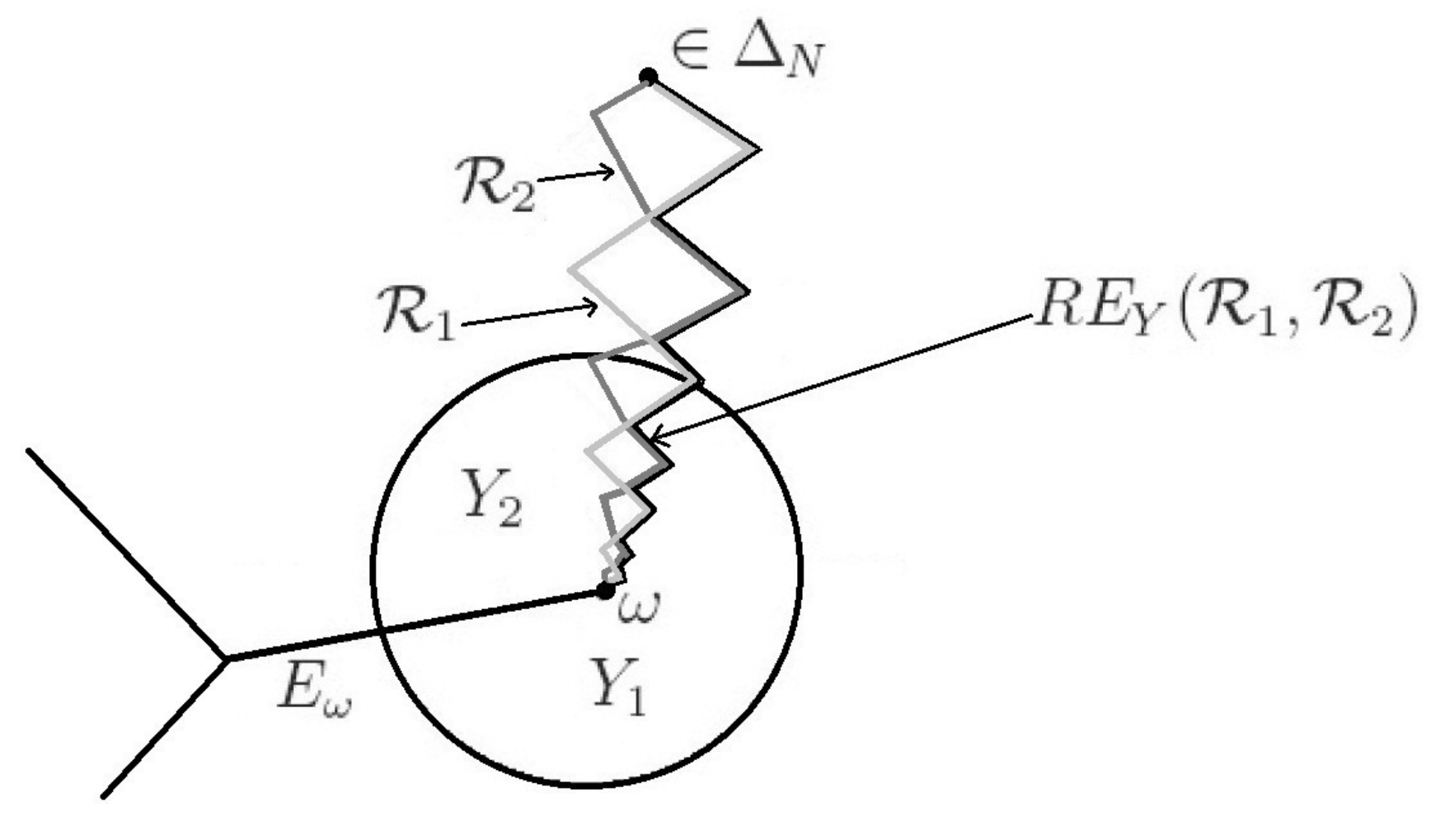}
\end{picture}
\caption{\label{Fig_RightEnvelope}  Two rays $\mathcal{R}_1$ and $\mathcal{R}_2$ that cross intersect and land at the point $\omega$ are indicated by light and medium gray.  The disk represents the set $Y$, and the right envelope $RE_Y(\mathcal{R}_1,\mathcal{R}_2)$ is indicated by the jagged black line.}
\end{center}
\end{figure}

\begin{remark}
Note that the construction of the Newton ray $RE(\mathcal{R}_1,\mathcal{R}_2)$ doesn't depend on the choice of $Y$. The right envelope $RE(\mathcal{R}_1,\mathcal{R}_2,\ldots,\mathcal{R}_n)$ of finitely many Newton rays $\mathcal{R}_1,\mathcal{R}_2,\ldots,\mathcal{R}_n$ is defined analogously.
\end{remark}

\begin{theorem} \label{Lem_RepPtsFixedRays} Let $N_p$ be a postcritically finite Newton map, and let $H$ be a Hubbard tree of period $m>1$ as constructed in Remark \ref{Rmk:HubFromRenorm}. For any repelling fixed point $\omega$ of $N_p^m$ in $H$, there exists a Newton ray of period $m\cdot \ell$ that lands at $\omega$, where $\ell$ is the period of the external rays landing at $\omega$ after straightening as in Theorem \ref{thm:straightening}. 
\end{theorem}
\begin{proof}
It follows from Theorem \ref{Lem_RepPtsAreLandingPts} that there exists a positive integer $r$ and a Newton ray $\mathcal{R}_{1}$ of period $mr$ that lands at $\omega$. Let $\mathcal{R}_{i} = N^{(i-1)\cdot m}_p(\mathcal{R}_{1})$ for $1\leq i \leq r$ and recall that the right envelope
\[
\mathcal{R} = RE(\mathcal{R}_{1},\mathcal{R}_{2},\ldots,\mathcal{R}_{r})
\]
 is a periodic Newton ray that lands at $\omega$. Denote by $Y$ the neighborhood of $\omega$ such that for some branch $h=N^{-m}_p, \, h(Y) \subset Y$ and let $Y_1$ be the connected component of 
\[
Y\setminus \bigcup_{i=1}^r \mathcal{R}_{i}
\]
such that 
\[
\mathcal{R} \cap \partial{Y_1} \neq \emptyset \quad \mbox{and} \quad E_\omega \cap \partial{Y_1} \neq \emptyset.
\]
Since the map $N_p$ is orientation preserving, 
\[
N^{m\cdot\ell}_p(Y_1) \cap Y_1 = Y_1,
\]
and because the $\mathcal{R}_i$ form a cycle, $N^{m\cdot\ell}_p(\mathcal{R}) = \mathcal{R} \cup \mathcal{E}$, where $\mathcal{E}$ is a union of edges of $\Delta_\N$. Therefore $RE(\mathcal{R}_{1},\mathcal{R}_{2},\ldots,\mathcal{R}_{r})$ is a Newton ray of period $m\cdot\ell$.
\end{proof}


\section{Extended Newton graphs from Newton maps} \label{Sec_ConstExtNewtGraph} 
\subsection{Construction}\label{Subsec_ConstructionExtNewtGraph}
In Section \ref{Sec_NewtonGraph}, the Newton graph of arbitrary level was constructed for a given postcritically finite Newton map $N_p$ (see Definition \ref{Def_ConcreteNewtonGraph}).    Next, extended Hubbard trees were constructed in Section \ref{Sec_RenormalizationNewton} to give a combinatorial description of the periodic free postcritical points.  Finally, periodic Newton rays were constructed in Section \ref{Sec_Newtrays} to connect the Hubbard trees to the Newton graph.

Here we specify the level of the Newton graph that will be used in the construction of the extended Newton graph. It is a consequence of Corollary \ref{Cor:PrepolesInNewtonGraph} that above a certain level, the Newton graph contains all critical points that are eventually fixed.

\begin{definition}[Newton graph of a postcritically finite Newton map] \label{Def_TheNewtonGraphPCFMap}
For a postcritically finite Newton map $N_p$, let $\N$ be the minimal integer such that 
\begin{itemize}
\item every pole is contained in $\Delta_\N$,
\item every critical point that eventually lands on the channel diagram $\Delta$ is contained in $\Delta_{\N}$, and
\item all periodic Hubbard trees and preperiodic trees are separated by $\Delta_\N$ as in Corollary \ref{Rem_ExtHubbardTrees}.

\end{itemize}The graph $\Delta_\N$ is called \emph{the Newton graph} of $N_p$.  
\end{definition}

    The proof of the following theorem uses these objects to construct a connected finite forward-invariant graph $\Delta^*_\N$ containing the postcritical set.  This graph will then be defined to be the extended Newton graph of $N_p$.

\begin{theorem}
\label{thm:ExtendNewtGraph} For a given postcritically finite Newton map $N_p$, let $\Delta_\N$ be the Newton graph of $N_p$. There exists a finite connected graph
$\Delta^*_\N$ that contains $\Delta_\N$, is invariant under $N_p$ and contains the critical and postcritical set of $N_p$. Furthermore, every edge of $\Delta^*_\N$ is eventually mapped by $N_p$ either into $\Delta_\N$, into an extended Hubbard tree, or onto a periodic Newton ray union edges from $\Delta_\N$. 
\end{theorem}

\begin{proof} 
The Newton graph $\Delta_\N$ captures the behavior of postcritical points of $N_p$ which eventually map into the channel diagram $\Delta$.  We now deal with the critical and postcritical points of $N_p$ which are not eventually fixed (recall that the set of such points is denoted $Q$). The graph will be constructed by pulling back invariant graphs, from which forward invariance under $N_p$ easily follows.

\textbf{Periodic Hubbard trees:} 
Let $H(q)$ be a periodic Hubbard tree as in Remark \ref{Rmk:HubFromRenorm} with lowest period $m$ (the dependence on $q$ is largely suppressed for the next two paragraphs). By Theorem \ref{Lem_RepPtsFixedRays}, there is a period $m\cdot\ell$ Newton ray $\gamma_H$ that lands at a repelling fixed point of $N_p^{m}$ in $H$.  Denote by $\tilde{\gamma}_H$ all rays in $N_p^{-1}(N_p(\gamma_H))$ that land on $H$.  We define 
\[
\Upsilon(H) = \big[\bigcup_{i=0}^{m\ell-1} N_p^{i}(H\cup\tilde{\gamma}_H) \big]\setminus \Delta_\N.
\]
Then $\Upsilon(H)\cup \Delta_\N$ is a connected forward invariant graph that is a union of Newton ray, Newton graph, and Hubbard tree edges.  All edges in the graph are disjoint, except possibly at their endpoints. This construction depends only on the tree $H(q)$, namely for any $p\in Q\cap H(q)$, one has $\Upsilon(H(p))=\Upsilon(H(q))$.

\textbf{Pre-periodic Hubbard trees:}
Now let $H'$ be a tree from Remark \ref{Rmk:HubFromRenorm} that is not periodic but $H=N_p(H')$ is a periodic Hubbard tree. Denote by $\gamma^{\infty}_H$ the extension of $\gamma_H$ by a simple path in $\Delta_\N$ ending at $\infty$. From each component of  $N_p^{-1}(\gamma^{\infty}_H)$ that lands at $H'$, extract a simple ray connecting $H'$ to the Newton graph (recall that all poles are contained in $\Delta_\N$ so each component of $N_p^{-1}(\gamma^{\infty}_H)$ will contain such a ray). Let $\tilde{\gamma}(H')$ be the union of all such simple rays landing at $H'$.  
Each such ray is evidently a pre-periodic Newton ray, and the rays may only intersect at endpoints since they were constructed by lifting. We define

\[
\Upsilon(H') := H' \cup \tilde{\gamma}_{H'}
\]
where it is evident from the construction that $N_p(\Upsilon(H'))=\Upsilon(H)$. Continuing inductively, define $\Upsilon(H')$ for all pre-periodic $H'$ that intersect $Q$.

\textbf{Construction of the graph:}
The graph satisfying the conclusion of the theorem is 
\begin{equation}\label{eqn:ExtendedNewtonGraph}
\Delta^*_\N = \Delta_\N \cup \bigcup_{q \in Q}
\Upsilon(H(q))
\end{equation}
where the notation $H(q)$ for periodic and preperiodic Hubbard trees was defined in Remark \ref{Rmk:HubFromRenorm}.

As constructed, an extended Newton graph $\Delta^*_\N$ with Newton ray edges may have infinitely many vertices since some Newton rays are composed of a sequence of infinitely many preimages of edges.  We now alter the edge set and vertex set of $\Delta^*_\N$ to produce a finite graph without changing the topology of $\Delta^*_\N$ as a subset of $\S^2$.  Each (periodic and pre-periodic) Newton ray is taken as a single edge, thereby eliminating all of the vertices in the Newton ray except its endpoints.  For convenience, we still denote this new finite graph by $\Delta^*_\N$.  Thus the vertices of $\Delta^*_\N$ are the vertices of $\Delta_\N$, the vertices of the Hubbard trees (which are chosen to include repelling fixed points of the polynomial-like restrictions and postcritical points of $N_p$ in the filled Julia sets), and points in the Hubbard tree preimages which map to these vertices.  This graph is finite, connected, and contains the whole postcritical set of $N_p$. Moreover, every edge of $\Delta^*_\N$ is evidently mapped by $N_p$ in the required way.
\end{proof}

\begin{definition}[Extended Newton graph]\label{defn_extendedNewtonGraph}
We call the pair $(\Delta^*_\N, N_p|_{\Delta^*_\N})$ from Equation \ref{eqn:ExtendedNewtonGraph} an \emph{extended Newton graph} associated to $N_p$.
\end{definition}

The following proposition asserts that the extended Newton graph assigned to a Newton map is unique on the Newton graph and Hubbard tree edges (though of course uniqueness is not expected for the Newton rays).  It is a consequence of Corollary \ref{Rem_ExtHubbardTrees} and the construction.

\begin{proposition}
 Let $(\Delta_{\N,1}^*,N_p)$ and $(\Delta_{\N,2}^*,N_p)$ be two extended Newton graphs constructed for $N_p$, and denote by $\Delta_{\N,1}^-$ and $\Delta_{\N,2}^-$ the respective graphs with all Newton ray edges removed.  Then $\Delta_{\N,1}^-=\Delta_{\N,2}^-$ and $N_p|_{\Delta_{\N,1}^-}=N_p|_{\Delta_{\N,2}^-}$.
\end{proposition}

\subsection{Example}\label{Subsec_ExamplesExtNewtGraphs}
There is a postcritically finite Newton map $N_p$ associated to a monic polynomial $p$ whose roots are given approximately by
\[a_1=1, a_2=-1,a_3=-0.0094672882+.3728674604i,\] \[a_4=-0.0094672882-.3728674604i\]
that satisfies the following: the roots of $p$ are simple critical points of $N_p$, and $N_p$ has two additional real critical points at $z\approx 0.3740835220,-.3835508102$ lying in a two cycle and a four cycle respectively. Figure \ref{Fig_Deg4ExtNewtonGraph1} displays the dynamical plane of $N_p$.

\begin{figure}[h!]
\begin{center}
\setlength{\unitlength}{1cm}
\begin{picture}(14,14)
\includegraphics[width=14cm]{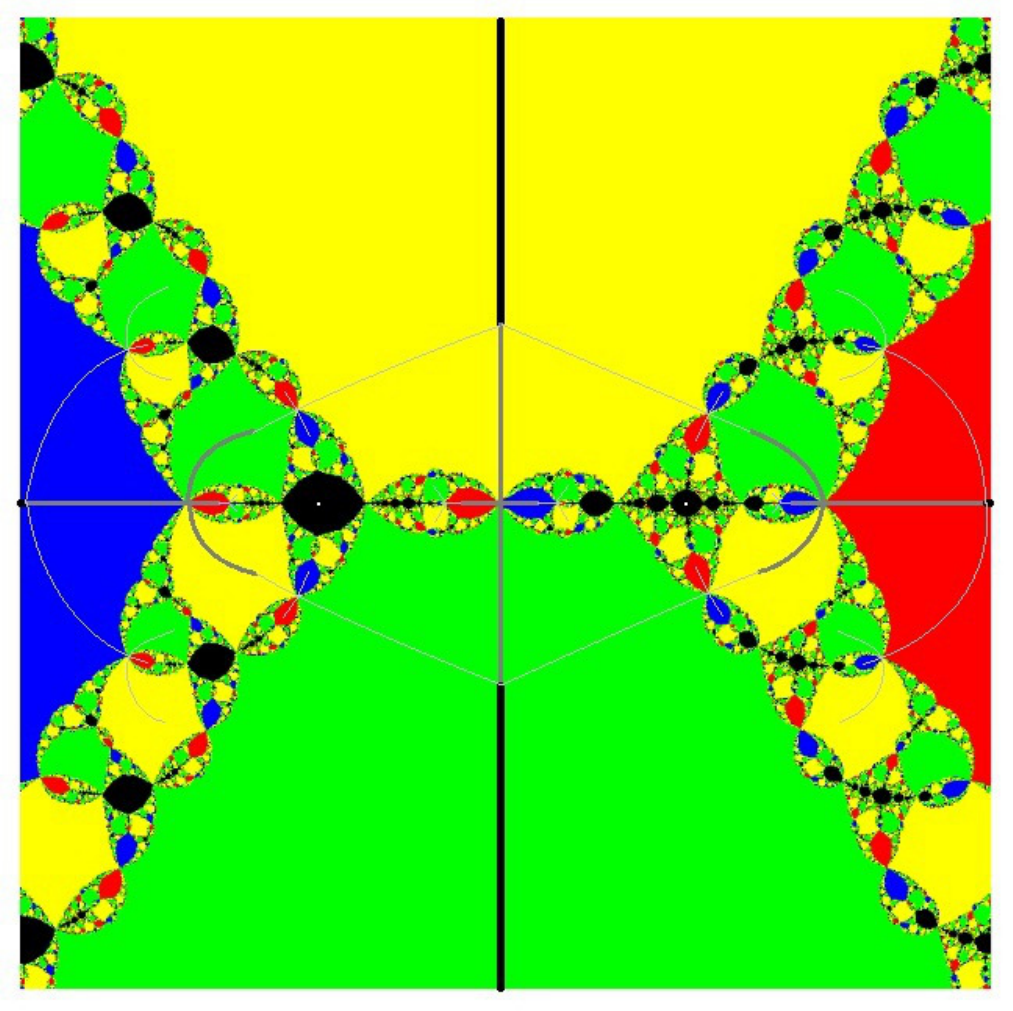}
\end{picture}
\caption{\label{Fig_Deg4ExtNewtonGraph1} Part of the dynamical plane for $N_p$ from Section \ref{Subsec_ExamplesExtNewtGraphs}.  The largest colored regions are the immediate basins of the roots.  The channel diagram is indicated by black lines.  The edges in $\Delta_1\setminus\Delta$ are indicated by thinner lines with a lighter shade, and the edges of $\Delta_2\setminus\Delta_1$ are indicated by even thinner and lighter lines.  The union of all three kinds of lines indicates the Newton graph.  Two white dots in the black filled Julia set indicate free critical points (image by K. Mamayusupov)}
\end{center}
\end{figure}

\begin{figure}[h!]
\begin{center}
\setlength{\unitlength}{1cm}
\begin{picture}(13,6)
\includegraphics[width=13cm]{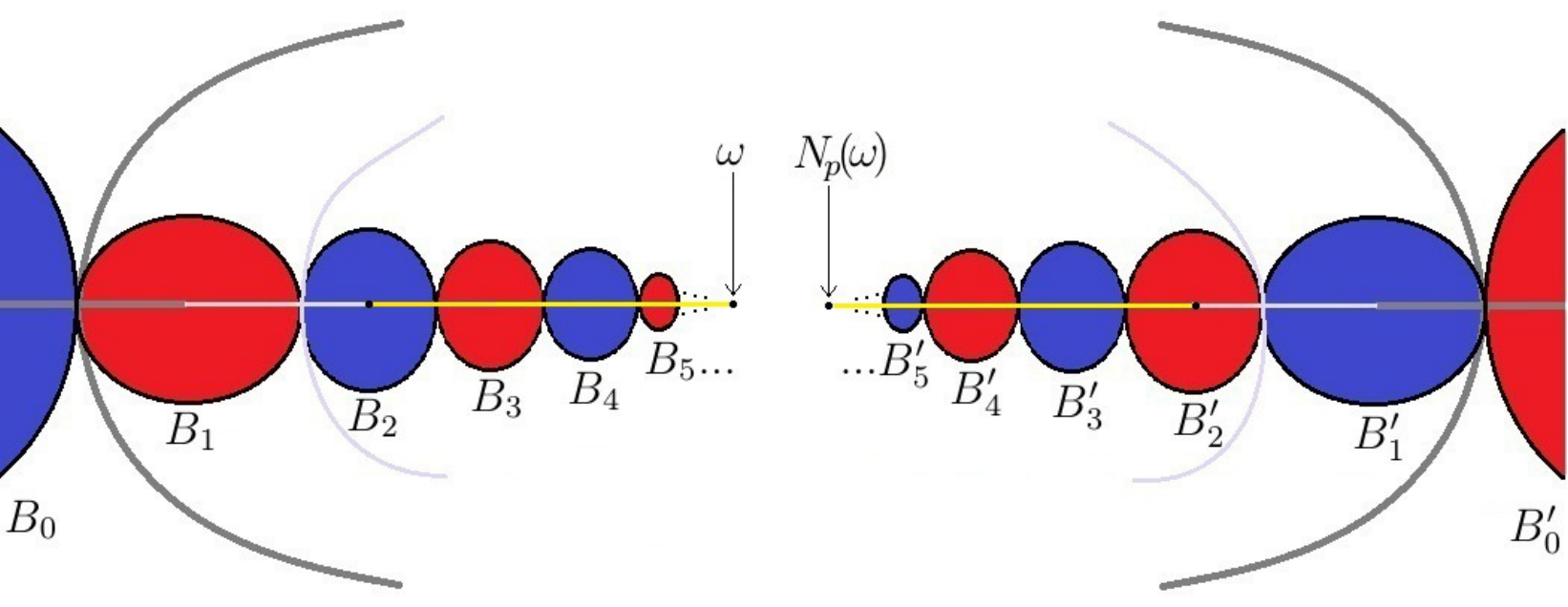}
\end{picture}
\caption{\label{Fig_NewtonDeg4BubbleRays} A topological model of the bubble and Newton rays for  landing at $\omega=N_p^2(\omega)$ and $N_p(\omega)$.  The bubbles $B_0$ and $B'_0$ correspond to the immediate basins of roots, and the dark grey edges correspond to edges in the Newton graph.  The light grey edges are the Newton rays, which have period 2.  For $i\neq 0$, $N_p(B_i)=B'_{i-1}$ and $N_p(B'_i)=B_{i-1}$}
\end{center}
\end{figure}

Since the polynomial $p$ has real coefficients, it is evident that $N_p$ must have a $z\mapsto \overline{z}$ symmetry.

The Newton graph of $N_p$ is taken to be the Newton graph of level two (see Definition \ref{Def_TheNewtonGraphPCFMap}).  

Renormalization (in the sense of Section \ref{Sec_RenormalizationNewton}) at either of the free simple critical points yields a degree four polynomial-like map.  The corresponding filled Julia sets each contain a simple critical point of $N_p$ and are otherwise mapped 2:1 onto each other by $N_p$.  The first-return map in the left-hand filled Julia set (which contains the critical point $-.3835508102$) has three fixed points.  The fixed points are given approximately by:
\begin{itemize}
\item the point $-.3835508102$ indicated by a white dot in the figure.
\item the left most endpoint of the filled Julia set $\omega\approx-0.5531911255$.
\item the unique point in the filled Julia set that lies in the closure of the immediate basins of the Newton map which contain non-real roots.
\end{itemize}

It is obvious how one might construct a periodic Newton ray for this last point: there is a period 2 ray in the yellow basin as well as a period 2 ray in the green basin that would suffice. In the interest of illustrating a more complex case,  we choose to connect $\omega$ and $N_p(\omega)$ to the Newton graph by periodic Newton rays that intersect infinitely many bubbles (see Figure \ref{Fig_NewtonDeg4BubbleRays}).  Denote by $B_0$ the immediate basin of the negative real root of $p$, and denote by $B'_0$ the immediate basin of the positive root.  Let $B_1$ be the unique preimage of $B'_0$ that is not an immediate basin and is adjacent to $B_0$, and define $B'_1$ similarly.  Inductively define $B_i$ to be the unique preimage of $B'_{i-1}$ that is adjacent to $B_{i-1}$, and define $B'_i$ similarly.  Note that for $i\neq 0$, we have $N_p(B_i)=B'_{i-1}$ and $N_p(B'_i)=B_{i-1}$.  Furthermore, the $B_i$ accumulate on $\omega$, and the $B'_i$ accumulate on $N_p(\omega)$. Let $\mathcal{B},\mathcal{B}'$ denote the bubble ray composed of the $B_i,B_i'$ respectively.  Then note that the corresponding Newton rays $\mathcal{R}(\mathcal{B}), \mathcal{R}(\mathcal{B}')$ form a two-cycle that connect the extended Hubbard trees of the filled Julia sets to the Newton graph.

The extended Newton graph $\Delta^*_\N$ is now defined as follows.  The vertices are the vertices of the Newton graph and the vertices of the two extended Hubbard trees containing the free critical points.  The edges are the edges of the Newton graph, the edges of the two extended Hubbard trees, and the two periodic Newton rays just constructed together with all preimages of the Newton rays that land on the two extended Hubbard trees.


\section{Abstract extended Newton graphs}\label{Sec_NewtMapsGenerateExtNewtGraphs}

In this section we define the abstract axiomatizations that describe the Newton graph together with its extension by Hubbard trees and Newton rays, and then we show that every postcritically finite Newton map indeed has extended Newton graphs that satisfy these axioms, as claimed in Theorem \ref{Thm_NewtMapsGenerateExtNewtGraphs}.  The converse that every abstract extended Newton graph is indeed realized by a postcritically finite map is true; this will be proved in \cite{LMS2}.  

Abstract Newton rays must first be defined. Let $\Gamma$ be a finite connected graph embedded in $\S^2$ and $f: \Gamma \to \Gamma$ a weak graph map.

\begin{definition}
\label{Def_AbstractNewtonRay}
\emph{Let $\mathcal{R}$ be an arc in $\S^2$ whose endpoints are denoted $i(\mathcal{R})$ and $t(\mathcal{R})$. Then $\mathcal{R}$ is called \emph{a periodic abstract Newton ray with respect to $(\Gamma,f)$}} if the following are satisfied: 
\begin{itemize}
\item  $\mathcal{R} \cap \Gamma=\{i(\mathcal{R})\}$.
\item there is a minimal positive integer $m$ so that $\overline{f}^m(\mathcal{R})= \mathcal{R} \cup \mathcal{E}$, where $\mathcal{E}$ is a (possibly empty) subgraph of $\Gamma$.
\end{itemize}
We say that the integer $m$ is \emph{the period} of $\mathcal{R}$, and that $\mathcal{R}$ \emph{lands} at $t(\mathcal{R})$. 
\end{definition}
\begin{definition}
\label{Def_AbstractNewtonRayPrePer}Let $\mathcal{R'}$ be an arc in $\S^2$ whose endpoints are denoted $i(\mathcal{R'})$ and $t(\mathcal{R'})$. Then $\mathcal{R'}$ is called 
a \emph{preperiodic abstract Newton ray with respect to $(\Gamma,f)$} the following hold:
\begin{itemize}
\item $\mathcal{R'}\cap\Gamma=\{i(\mathcal{R'})\}$. 
\item there is a minimal integer $l>0$ such that $\overline{f}^l(\mathcal{R'}) = \mathcal{R} \cup \mathcal{E}$, where $\mathcal{E}$ is a (possibly empty) subgraph of $\Gamma$ and $\mathcal{R}$ is a periodic abstract Newton ray with respect to $(\Gamma,f)$.
\item $\mathcal{R'}$ is not a periodic abstract Newton ray with respect to $(\Gamma,f)$.
\end{itemize}
We say that the integer $l$ is \emph{the preperiod} of $\mathcal{R}'$, and that $\mathcal{R}'$ \emph{lands} at $t(\mathcal{R}')$.\end{definition}

Now we will introduce the concept of an abstract extended Newton graph.  In \cite{LMS2}, this graph will be shown to carry enough information to characterize postcritically finite Newton maps.

\begin{definition}[Abstract extended Newton graph]
\label{Def_AbstractExtNewtGraph} 
Let $\Sigma \subset \S^2$ be a finite connected graph, and let $f:\Sigma \to \Sigma$ be a weak graph map.  A pair $(\Sigma,f)$ is called an \emph{abstract extended Newton graph} if
the following are satisfied:

\begin{enumerate}
\renewcommand{\theenumi}{(\arabic{enumi})}

\item \label{Cond_NewtonGrah} (Abstract Newton graph) There exists a positive integer $\N$ and an abstract Newton graph
$\Gamma$ at level $\N$ so that $\Gamma \subseteq \Sigma$. Furthermore $\N$ is minimal so that condition \ref{Cond_TreesSeparated} holds.

\item \label{Cond_PerHubbardTrees} (Periodic Hubbard trees) There is a finite collection of  (possibly degenerate) minimal abstract extended Hubbard trees $H_i \subset \Sigma$ which are disjoint from $\Gamma$, and for each $H_i$ there is a minimal positive
integer $m_i\geq 2$ called the \emph{period of the tree} such that $f^{m_i}\left(H_i\right)=H_i.$

\item \label{Cond_PrePerHubbardTrees} (Preperiodic trees) There is a finite collection of possibly degenerate trees 
$H'_i \subset \Sigma$ of preperiod $\ell_i$, i.e. there is a minimal positive integer $\ell_i$ so that $f^{\ell_i}(H'_i)$ is a periodic Hubbard tree ($H'_i$ is not necessarily a Hubbard tree).  Furthermore for each $i$, the tree $H'_i$ contains a critical or postcritical point.

\item \label{Cond_TreesSeparated} (Trees separated) Any two different periodic or pre-periodic Hubbard trees lie in different complementary components of $\Gamma$.

\item \label{Cond_NewtonRaysPer}
(Periodic Newton rays) For every periodic abstract extended Hubbard tree $H_i$ of period $m_i$, there is exactly one periodic abstract Newton ray $\mathcal{R}_i$ with respect to $(\Gamma,f)$.  The ray lands at a repelling fixed point $\omega_i\in H_i$ of $f^{m_i}$ and has period $m_i\cdot r_i$ where $r_i$ is the period of any external ray landing at the corresponding fixed point of the polynomial realizing $H_i$.

\item \label{Cond_NewtonRaysPreper}
(Preperiodic Newton rays) For every preperiodic tree $H'_i$, there exists at least one preperiodic abstract Newton ray in $\Sigma$ with respect to $(\Gamma,f)$ connecting a vertex of $H'_i$ to $\Gamma$.

\item \label{Cond_uniqueExtend} (Unique extendability) For every vertex $y\in V(\Sigma)$ and every component $U$ of $\S^2\setminus\Sigma$, the local extension $\tilde{f}$ from Equation (\ref{eqn:localExtension}) is injective on
$\bigcup_{v\in f^{-1}(y)} U_v \cap U$.

\item \label{Cond_DegreeExtGraph} (Topological admissibility) $\sum_{x\in V(\Sigma)} \left(\deg_x f-1\right) = 2d_{\Gamma}-2$, where $d_{\Gamma}$ is the degree of the abstract channel diagram $\Delta \subset \Gamma$.

\item \label{Cond_TypesOfEdges} (Edges and vertices) Every edge in $\Sigma$ must be one of the following three types:
\begin{itemize}
\item Type N: An edge in the abstract Newton graph $\Gamma$ of condition \ref{Cond_NewtonGrah}.
\item Type H: An edge in a periodic or pre-periodic abstract Hubbard tree of condition \ref{Cond_PerHubbardTrees} or \ref{Cond_PrePerHubbardTrees}.
\item Type R: A periodic or pre-periodic abstract Newton ray with respect to $(\Gamma,f)$ from condition \ref{Cond_NewtonRaysPer} or \ref{Cond_NewtonRaysPreper}.
\end{itemize}
As a consequence, every vertex of $\Sigma$ is either a Hubbard tree vertex or a Newton graph vertex.

\end{enumerate}

\end{definition}

\begin{remark}[Regular extension]
 The purpose of including condition \ref{Cond_uniqueExtend} is that after $f$ has been upgraded to a graph map following Remark \ref{rem_weakGraphMap}, the hypothesis of Proposition \ref{Prop_RegExt} is met. Thus $f$ has a regular extension $\ol{f}$ which is unique up to Thurston equivalence.
\end{remark}

Now we are going to give the proof of our main theorem which states that an extended Newton graph of a postcritically finite Newton map is indeed an abstract extended Newton graph.

\begin{proof}[Proof of Theorem \ref{Thm_NewtMapsGenerateExtNewtGraphs}]

For a given Newton map $N_p$ consider the extended Newton graph $\Delta^*_\N$ from Definition \ref{defn_extendedNewtonGraph}. We show that $(\Delta^*_\N,N_p)$ is an abstract extended Newton graph by verifying all nine conditions of Definition \ref{Def_AbstractExtNewtGraph}.

\ref{Cond_NewtonGrah} Let $\Delta_\N$ be the Newton graph of $N_p$ as in Definition \ref{Def_TheNewtonGraphPCFMap}. Then $(\Delta_\N,N_p)$ satisfies the properties of an abstract Newton graph by Theorem \ref{Thm_NewtonGraph}.  Minimality is immediate.

\ref{Cond_PerHubbardTrees} The periodic extended Hubbard trees constructed in Theorem \ref{thm:ExtendNewtGraph} evidently satisfy the properties of abstract extended Hubbard trees (Theorem \ref{Thm_ExtendedHubbardTreesBijectiveCorr}). Corollary \ref{Rem_ExtHubbardTrees} states that there is no common vertex with the Newton graph.

\ref{Cond_PrePerHubbardTrees} Also by construction of $\Delta^*_\N$, the preperiodic trees are preimages of periodic Hubbard trees under iterates of $N_p$. Since periodic Hubbard trees may not intersect the Newton graph, the preimage trees may have no common vertex with the Newton graph.

\ref{Cond_TreesSeparated} The existence of a level of the Newton graph so that trees are separated is due to Corollary \ref{Rem_ExtHubbardTrees}.

\ref{Cond_NewtonRaysPer},\,\ref{Cond_NewtonRaysPreper} Every periodic Newton ray (see Definition \ref{Def_NewtonRay}) is easily shown to be a periodic abstract Newton ray, and the corresponding statement holds for preperiodic rays.  The rest of the properties follow immediately from the construction.

\ref{Cond_uniqueExtend} The level of the Newton graph $\N$ was chosen so that $\Delta_\N$ contains the poles and eventually fixed critical points of $N_p$.  Clearly $\Delta_\N=N_p^{-1}(\Delta_{\N-1})$, and so $N_p$ is injective on the complementary components of $\Delta_\N$ that contain no free critical points. Now let $U$ be a complementary component of $\Delta_\N$ that contains at least one critical point of $N_p$. All critical points in $U$ must be free, and by the construction in Theorem \ref{thm:ExtendNewtGraph} are contained in a single Hubbard tree or a preperiodic tree which we denote $H$. Recalling that $\tilde{\gamma}_H=N_p^{-1}(N_p(\gamma_H))$, it is clear that $N_p$ is injective on complementary components of $U\setminus\tilde{\gamma}_H$. The rays in $\tilde{\gamma}_H$ are also edges in $\Delta^*_\N$. We conclude that $N_p$ is injective on all complementary components of $\Delta^*_\N$, and so the condition is satisfied.

\ref{Cond_DegreeExtGraph} Since the degree of $N_p$ equals the degree of its channel diagram, the conclusion follows from the Riemann-Hurwitz formula.

\ref{Cond_TypesOfEdges}  By construction, every edge of $\Delta^*_\N$ is either type N, H, or R.   We show that the edges of each type may intersect only over vertices.  By Corollary \ref{Rem_ExtHubbardTrees}, type H edges may not intersect type N edges, and by construction the intersections with other type H edges may only be over vertices.  It follows from Remark \ref{Rem_RaysDisjointFromHubbardTrees} that the interiors of edges of type H are also disjoint from edges of type R. By Definition \ref{Def_NewtonRay}, the edges of type N and edges of type R can only intersect at vertices of $\Delta_\N$.   Finally, given a type R edge $\gamma$ in a complementary component $U$ of the Newton graph, the only other type R edges constructed in $U$ are the Newton rays in $N_p^{-1}(N_p(\gamma))$. These rays clearly have pairwise disjoint interiors since $\gamma$ contains no critical value in its interior. 
\end{proof}

\section{Conclusion}

We have shown how to extract a graph from any postcritically finite Newton map that satisfies the defining properties of an abstract extended Newton graph.  In \cite{LMS2}, it will be shown that every abstract extended Newton graph is realized by a postcritically finite Newton map.  An equivalence relation will be placed on the set of all abstract extended Newton graphs, and it will be shown that there is a bijection between equivalence classes and the postcritically finite Newton maps up to affine conjugacy.  This will complete the combinatorial classification of postcritically finite Newton maps.


\bibliography{main1}
\bibliographystyle{alpha}


\end{document}